\documentclass[11pt]{article}
\usepackage[T1]{fontenc}
\usepackage{amsmath,amsfonts,amssymb,theorem,color}
\usepackage{mathbbol}

\theoremstyle{change} 

\newenvironment{proof}{\noindent \bf Proof:   \rm}{\hspace*{\fill}
$\square$ \medskip}

{\theorembodyfont{\itshape}
\newtheorem{lemma}{Lemma}
\newtheorem{proposition}[lemma]{Proposition}

\newtheorem{corollary}[lemma]{Corollary}
\newtheorem{theorem}[lemma]{Theorem}}
{\theorembodyfont{\upshape}

\newtheorem{definition}[lemma]{Definition}
\newtheorem{remark}[lemma]{Remark}

\newtheorem{example}[lemma]{Example}
}

\DeclareMathAlphabet{\mathpzc}{OT1}{pzc}{m}{it}

\usepackage[all,cmtip]{xy}

\usepackage{appendix}

\usepackage{graphicx}

    \mathchardef\mhyphen="2D

\newcommand{\eq}[1][r]
   {\ar@<-3pt>@{-}[#1]
    \ar@<-1pt>@{}[#1]|<{}="gauche"
    \ar@<+0pt>@{}[#1]|-{}="milieu"
    \ar@<+1pt>@{}[#1]|>{}="droite"
    \ar@/^2pt/@{-}"gauche";"milieu"
    \ar@/_2pt/@{-}"milieu";"droite"}

\begin{document}

\title{Multipliers of a semigroup object in a monoidal category}   
\author{Laurent Poinsot\\
\small CREA, \'Ecole de l'Air et de l'Espace, Salon-de-Provence and\\
\small LIPN, UMR CNRS 7030, {Universit\'e Sorbonne Paris Nord, Villetaneuse, France. }\\ 
{\tt \small poinsotlaurent@gmail.com}}
 
 \date{}                                  


\maketitle

\begin{abstract}
The monoid of  multipliers  of a semigroup object  in a monoidal category is introduced, arising from an abstraction of the  definition of the translational hull of an ordinary semigroup or of the multiplier algebra of a Banach algebra and dually, the monoid of comultipliers of a cosemigroup object is obtained. Its set-theoretic version, the classical translational hull, is shown to provide a functor from a subcategory of ordinary semigroups to monoids, similar to a left adjoint. The abstract multiplier monoid of a semigroup object is related to the concrete translational hull of its convolution semigroup by a ``concretization'' homomorphism. For semigroup objects for which this homomorphism is onto, the multiplier construction is functorial and the concretization homomorphisms form a natural epimorphism. 
\vskip 6pt
\noindent
 {\bf MSC 2020}:  {18C40, 20M30,  18M05}  \newline
\vskip 1pt
\noindent {\bf Keywords:} {Semigroup object,  multiplier, 
translational hull, convolution semigroup.  }
\end{abstract}

\section{Introduction}

Let $\mathsf{S}=(S,*)$ be a set-theoretic semigroup. By a left (resp. right) translation is meant a map $L$ (resp. $R$) from $S$ to itself such that $L(x*y)=L(x)*y$ (resp., $R(x*y)=x*R(y)$). Let $\mathit{LTr}(\mathsf{S})$ and $\mathit{RTr}(\mathsf{S})$ be respectively the set of all left and of all right translations of $\mathsf{S}$. For instance given $x\in S$, $\mathfrak{L}_x\colon S\to S$, $y\mapsto x*y$ and $\mathfrak{R}_x\colon y\mapsto y*x$ are respectively left and right translations. Associativity of $*$ provides for such {\em inner} translations, the following {\em linking} relation: $\mathfrak{R}_x(y)*z=y*\mathfrak{L}_x(z)$, $x,y,z\in S$. More generally a pair $(L,R)$ consisting of a left and a right translations such that for each $y,z\in S$, $R(y)*z=y*L(z)$ is referred to as a {\em multiplier} of $\mathsf{S}$. The set $\mathit{TrHull}(\mathsf{S})$ of all multipliers of $\mathsf{S}$ is called the {\em translational hull} of $\mathsf{S}$ (see e.g.~\cite{Clifford}), and in general it contains over multipliers than only the inner ones, that is, those of the form $\mathfrak{M}_{\mathsf{S}}(x)=(\mathfrak{L}_{x},\mathfrak{R}_x)$, $x\in S$. Multipliers may be composed by $(L',R')\star (L,R):=(L'\circ L,R\circ R')$ and $(id_S,id_S)$ acts as an identity for $\star$, making $\mathsf{TrHull}(\mathsf{S}):=(\mathit{TrHull}(\mathsf{S}),\star,(id_S,id_S))$ a monoid, and $\mathfrak{M}_{\mathsf{S}}\colon \mathsf{S}\to |\mathsf{TrHull}(\mathsf{S})|$, $x\mapsto (\mathfrak{L}_x,\mathfrak{R}_x)$ a homomorphism of semigroups, called the {\em canonical homomorphism} of $\mathsf{S}$, where $|-|\colon \mathbf{Mon}\to\mathbf{Sem}$ is the obvious forgetful functor. In fact,  $\mathit{TrHull}(\mathsf{S})$ is a submonoid of the product monoid $\mathsf{LTr}(\mathsf{S})\times \mathsf{RTr}(\mathsf{S})^{op}$, with $\mathsf{LTr}(\mathsf{S}):=(\mathit{LTr}(\mathsf{S}),\circ,id_S)$ and $\mathsf{RTr}(\mathsf{S}):=(\mathit{RTr}(\mathsf{S}),\circ,id_S)$, where for a monoid $\mathsf{M}=(M,*,1)$, $\mathsf{M}^{op}:=(M,*^{op},1)$ denotes  its {\em opposite} with $x*^{op}y:=y*x$, $x,y\in M$. 

The following result is easy and shows that the translational hull is not very interesting for monoids, what we could have suspected by thinking of it as a kind of unitarization. 
\begin{lemma}\label{lem:can_hom_is_iso_for_mon}
Let $\mathsf{M}=(M,*,1)$ be a monoid. Then,  the canonical homomorphism  of semigroups $\mathfrak{M}_{|\mathsf{M}|}\colon (M,*)\to |\mathsf{TrHull}(|\mathsf{M}|)|$ lifts to an isomorphism of monoids $\mathfrak{M}_{\mathsf{M}}\colon \mathsf{M}\to \mathsf{TrHull}(|\mathsf{M}|)$, that is, $\mathfrak{M}_{\mathsf{M}}$ is an isomorphism, and $|\mathfrak{M}_{\mathsf{M}}|=\mathfrak{M}_{|\mathsf{M}|}$.
\end{lemma}

Historically multipliers appeared naturally with the study of ideal extensions of semigroups~\cite{Petrich} but dit not remain confined to the theory of ordinary semigroups. For instance the Banach algebra of multipliers of a commutative  (non-unital)  Banach algebra was considered in~\cite{Wang} and the $C^*$-algebra of multipliers of a (non-unital) $C^*$-algebra was studied in~\cite{Busby}. 

$\mathsf{S}$ acts on the set $S$ by left (resp. right) translations. Left (right) multipliers then are in fact nothing but endomorphisms of the left (resp. right) $\mathsf{S}$-act $S$ itself.  Now given a left (resp. right) translation $L$ (resp. $R$), $\mathsf{g}(L)\colon S\times S\to S$, $(x,y)\mapsto x*L(y)$ (resp. $\mathsf{d}(R)\colon S\times S\to S$, $(x,y)\mapsto R(x)*y$) is a homomorphism of left (resp. right) $\mathsf{S}$-acts from $S\times S$ to $S$, where $\mathsf{S}$ acts on $S\times S$ by left (resp. right) multiplication on the first (resp. second) factor. $(L,R)$ then is a multiplier iff $\mathsf{g}(L)=\mathsf{d}(R)$, that is, the translational hull of $\mathsf{S}$ is nothing but the pullback of $\mathsf{g}$ along $\mathsf{d}$. Such observations are not only true for multipliers of semigroups, but also for those of rings or Banach algebras. These  algebraic structures have in  common the fact  that they are all {\em semigroup objects} in some monoidal category: that of sets for the ordinary semigroups, of abelian groups for the rings, and Banach spaces for Banach algebras. This suggests to use  the language of monoidal categories to describe multipliers more abstractly. 

In doing so we are left with at least two possible ways. One way is to ``internalize'' the multiplier construction, as in~\cite{Bohm}. It requires the assumptions that the monoidal category into consideration is also  closed and has some specific pullbacks (in order to define an internal ``multiplier monoid object''), or when it is not closed, to consider  what is called therein ``$\mathbb{M}$-morphisms'' as a substitute for morphisms into an internal multiplier monoid object that may not exist. Without asking for extra assumptions but having in mind the existence of a multiplier monoid, another way  is to mimic the construction of the translational hull by considering abstract endomorphisms of a semigroup object $\mathsf{S}$ in a monoidal category $\mathbb{C}$, seen as a left and a right act on itself. In doing so naturally arises an ordinary monoid which is called the {\em multiplier monoid} $\mathsf{Mult}_{\mathbb{C}}(\mathsf{S})$ of $\mathsf{S}$ (see Section~\ref{sec:mult_mon}), and which reduces to the usual translational hull for a usual semigroup.  This alternative approach is consistent with the general philosophy adopted in this note that any semigroup object in any monoidal category may be ``concretized'' as an ordinary semigroup and finds its source in the original motivation of the author to study the functorial behavior of the ordinary translational hull construction.

In fact each monoidal category  $\mathbb{C}$ naturally comes together with a canonical  monoidal functor into the cartesian category of sets, namely the ``convolution functor'' which assigns to an object $c$ of the underlying category $C$ of $\mathbb{C}$, the set $C(I,c)$ of morphisms from the monoidal unit $I$ into this object (its so-called ``generalized elements''). As any monoidal functor, the convolution functor lifts to a functor from the category of semigroup objects of the domain monoidal category into that of the codomain category,  and so assigns  an ordinary semigroup, its ``concretization'' $\mathsf{Conv}_I(\mathsf{S})$ to an abstract semigroup $\mathsf{S}$. The generalized elements also play an important role in that they induce inner multipliers (see Section~\ref{sec:inner_mult_and_gen_elt}). 
A {\em concretization homomorphism} then relates $\mathsf{Mult}_{\mathbb{C}}(\mathsf{S})$ with $\mathsf{TrHull}(\mathsf{Conv}_I(\mathsf{S}))$ (Definition~\ref{def:concretization_morphism}). Section~\ref{sec:functoriality} is entirely devoted to the study this convolution construction.

The functorial behavior of both $\mathsf{TrHull}$ and $\mathsf{Mult}$ is explored in Section~\ref{sec:around_ord_sem}. More precisely   it is shown that under some conditions,  any homomorphism of semigroups $\mathsf{S}\xrightarrow{f}|\mathsf{TrHull}(\mathsf{T})|$, where $\mathsf{S},\mathsf{T}$ are ordinary semigroups, admits a unique ``extension'' $\mathsf{TrHull}(\mathsf{S})\xrightarrow{f^\sharp}\mathsf{TrHull}(\mathsf{T})$ as a homomorphism of monoids (Theorem~\ref{thm:extension}). This result is used to define a {\em translational hull functor} from a subcategory of semigroup  to monoids, which is very close to be a left adjoint to the forgetful functor from monoids to semigroups (Theorem~\ref{thm:close_to_left_adjoint}). In Section~\ref{sec:around_ord_sem_2}, using the notion of {\em concrete semigroup object}, that is, a semigroup object whose concretization homomorphism is onto, it is proved that under some conditions, any homomorphism of semigroups $\mathsf{Conv}_{I}(\mathsf{S})\xrightarrow{f}|\mathsf{Mult}_{\mathbb{D}}(\mathsf{T})|$, where $\mathsf{S},\mathsf{T}$ are semigroup objects in possibly different monoidal categories $\mathbb{C},\mathbb{D}$, has a unique ``extension'' as a homomorphism of monoids $\mathsf{Mult}_{\mathbb{C}}(\mathsf{S})\xrightarrow{f^M}\mathsf{Mult}_{\mathbb{D}}(\mathsf{T})$  (Theorem~\ref{thm:extension_for_internal_multipliers}). This result then is used to provide a {\em multiplier monoid functor} from a subcategory of semigroup objects in {\em any} monoidal category to monoids, and turns the concretization homomorphisms into a natural epimorphism  between this functor and the composition of the translational hull functor  by the convolution semigroup functor.

\section{Notations and prerequisites}\label{sec:prerequisites}

We next introduce some notions which could be unfamiliar for some readers.

\paragraph{Some notations:}

Let $C$ be a category and let $c,c'$ be $C$-objects. $C(c,c')$ denotes the hom-sets of $C$-morphisms with domain $c$ and codomain $c'$ and given a functor $F\colon C\to D$, if needed one denotes by $F_{c,c'}$ the hom-component $C(c,c')\to D(Fc,Fc')$ of $F$. $\mathbf{Set}$ and $\mathbf{Cat}$ are the categories of sets (and maps) and of categories (and functors). Given functors $F,G\colon C\to D$,  that $\alpha$ is a natural transformation from $F$ to $G$ is denoted by $\alpha\colon F\Rightarrow G\colon C\to D$ and $\mathit{Nat}(F,G)$ stands for the set of all such natural transformations.


\paragraph{Monoidal categories and monoidal functors:} A monoidal category is denoted $\mathbb{C}=({C},-\otimes-,I,\alpha,\lambda,\rho)$ or sometimes simply $\mathbb{C}=(C,\otimes,I)$, with $-\otimes-\colon C\times C\to C$ the tensor product bifunctor, $I$ the unit object and the natural isomorphisms $((a\otimes b)\otimes c\xrightarrow{\alpha_{a,b,c}}a\otimes(b\otimes c))_{a,b,c}$, $(I\otimes c\xrightarrow{\lambda_c}c)_{c}$ and $(c\otimes I\xrightarrow{\rho_c}c)_c$ which are coherence constraints of associativity, and of left and right unit, respectively. E.g. $\mathbb{Set}$ is the  monoidal category of sets under the cartesian product and initial object $1:=\{\, 0\,\}$ with $\mathbf{Set}$ as its underlying category. 

Let $\mathbb{C}$ be a monoidal category. Then one may define its {\em opposite} $\mathbb{C}^{op}:=(C^{op},\otimes^{op},I,\alpha^{-1},\lambda^{-1},\rho^{-1})$ where $C^{op}$ is the usual opposite category of $C$, and $\otimes^{op}\colon C^{op}\times C^{op}=(C\times C)^{op}\to C^{op}$ is the usual opposite of $\otimes$. 

One also defines the {\em tranpose} $\mathbb{C}^t=(C,\otimes^t,I,\alpha^t,\lambda^t,\rho^t)$ of $\mathbb{C}$ with $(c\xrightarrow{f}d)\otimes^t(c'\xrightarrow{f'}d'):=c'\otimes c\xrightarrow{f'\otimes f}d'\otimes d$, that is, $\otimes^t\colon C\times C\simeq C\times C\xrightarrow{\otimes}C$,  where the isomorphism is given by $((a,b)\xrightarrow{(f,g)}(c,d))\mapsto ((b,a)\xrightarrow{(g,f)}(d,c))$, $\alpha^t_{a,b,c}\colon (a\otimes^tb)\otimes^tc=c\otimes(b\otimes c)\xrightarrow{\alpha_{c,b,a}^{-1}}(c\otimes b)\otimes a=a\otimes^t(b\otimes^t c)$, $\lambda_a^t=I\otimes^t a=a\otimes I\xrightarrow{\rho_a}a$, $\rho_a^{t}=a\otimes^t I=I\otimes a\xrightarrow{\lambda_a}a$. 

Let $\mathbb{C}:=(C,\otimes_C,I_C,\alpha,\lambda,\rho)$ to $\mathbb{D}=(D,\otimes_D,I_D,\alpha,\lambda,\rho)$ be monoidal categories. A {\em lax monoidal functor} from $\mathbb{C}$ to $\mathbb{D}$ is a triple $\mathbb{F}:=(F,F^{(2)},F_0)$ where $F\colon C\to D$ is a functor (the {\em underlying functor} of $\mathbb{F}$), $(F(c_1)\otimes_D F(c_2)\xrightarrow{F^{(2)}_{c_1,c_2}}F(c_1\otimes_C c_2))_{c_1,c_2}$  is a natural transformation and $F_0\colon J\to FI$  is a $D$-morphism, subject to coherence conditions (see~\cite{Street} for details). A lax monoidal functor is called {\em strong monoidal} (resp. {\em strict monoidal}) when $F^{(2)}$ and $F_0$ are isomorphisms (resp. identities). 

Let $\mathbb{C}=(C,\otimes_C,I_C)\xrightarrow{\mathbb{F}=(F,F^{(2)},F_0)}\mathbb{D}=(D,\otimes_D,I_D)\xrightarrow{\mathbb{G}=(G,G^{(2)},G_0)}\mathbb{E}=(E,\otimes_E,I_E)$ be monoidal functors. Let $\mathbb{G}\circ\mathbb{F}:=(G\circ F,(G\circ F)^{(2)},(G\circ F)_0)$ with $(G\circ F)^{(2)}_{c,d}:=G(F_{c,d}^{(2)})\circ G^{(2)}_{F(c),F(d)}$ and $(G\circ F)_0:=G(F_0)\circ G_0$.  
With identity $\mathbb{id}_{\mathbb{C}}$ at $\mathbb{C}$ the strict monoidal functor with underlying functor $id_C$, this provides the category $\mathbf{MonCat}$ of monoidal categories and lax monoidal functors. 

\paragraph{Semigroup objects:}  
A {\em semigroup object} (or simply {\em semigroup}) in $\mathbb{C}$ is a pair $(S,S\otimes S\xrightarrow{\mu}S)$ such that the diagram below on the left 
\begin{equation}
\xymatrix@R=1em{
\ar[d]_{\alpha_{S,S,S}}(S\otimes S)\otimes S\ar[rr]^{\mu\otimes id_S} & &S\otimes S\ar[d]^{\mu}&&S\otimes S\ar[r]^{\mu}\ar[d]_{f\otimes f} & S\ar[d]^{f}\\
S\otimes (S\otimes S)\ar[r]_-{id_S\otimes \mu}&S\otimes S \ar[r]_{\mu}&S&&S'\otimes S'\ar[r]_{\mu'} & S'\\
}
\end{equation}
commutes, while a {\em semigroup morphism} $(S,\mu)\xrightarrow{f}(S',\mu')$ is a $C$-morphism $f\colon S\to S'$ making the above right diagram commutative. This defines the category $\mathbf{Sem}\mathbb{C}$ of semigroups in $\mathbb{C}$. $\mathbf{Sem}:=\mathbf{Sem}(\mathbb{Set})$ is the category of ordinary semigroups and their homomorphisms. The category $\mathbf{Cosem}\mathbb{C}$ of {\em cosemigroup objects} (or simply {\em cosemigroups}) in $\mathbb{C}$ is defined to be $(\mathbf{Sem}(\mathbb{C}^{op}))^{op}$. Let $(\mathbb{C},\sigma)$ be a symmetric monoidal category. Given a semigroup $\mathsf{S}=(S,\mu)$ (resp. a cosemigroup $\mathsf{S}=(S,\delta)$) one has its {\em opposite} (co)semigroup $\mathsf{S}^{op}:=(S,\mu\circ \sigma_{S,S})$ (resp. $\mathsf{S}^{op}:=(S,\sigma_{S,S}\circ \delta)$). A semigroup (resp. cosemigroup) $(S,\mu)$ (resp. $(S,\delta)$) is said to be {\em commutative}) (resp. {\em cocommutative}) when $\mathsf{S}=\mathsf{S}^{op}$. By ${}_{c}\mathbf{Sem}(\mathbb{C})$ and ${}_{coc}\mathbf{Cosem}(\mathbb{C})$ we denote the full subcategories of $\mathbf{Sem}(\mathbb{C})$ and $\mathbf{Cosem}(\mathbb{C})$ respectively spanned by the (co)commutative (co)semigroups. One has ${}_{coc}\mathbf{Cosem}\mathbb{C}=({}_c\mathbf{Sem}(\mathbb{C}^{op}))^{op}$. The next result is clear.
\begin{lemma}\label{lem:sem_in_C_is_the_same_sem_in_transpose}
$\mathbf{Sem}(\mathbb{C})=\mathbf{Sem}(\mathbb{C}^t)$ and dually $\mathbf{Cosem}(\mathbb{C})=\mathbf{Cosem}(\mathbb{C}^t)$. 
\end{lemma}

Given a lax monoidal functor $\mathbb{F}\colon \mathbb{C}\to\mathbb{C}'$,  $\widetilde{\mathbb{F}}(S,\mu):=(FS,FS\otimes' FS\xrightarrow{F^{(2)}_{S,S}}F(S\otimes S))$ and $\widetilde{\mathbb{F}}(f):=Ff$ define an induced functor $\widetilde{\mathbb{F}}\colon \mathbf{Sem}\mathbb{C}\to\mathbf{Sem}\mathbb{C}'$, such that the diagram
\begin{equation}
\xymatrix@R=1em{
\mathbf{Sem}\mathbb{C}\ar[d]_{U_s}\ar[r]^{\widetilde{\mathbb{F}}}&\mathbf{Sem}\mathbb{C}'\ar[d]^{U'_s}\\
C \ar[r]_{F}& C'
}
\end{equation}
commutes, with $U_s,U'_s$ the obvious forgetful functors. In fact this provides a functor $\mathit{Sem}$ from $\mathbf{MonCat}$ to $\mathbf{Cat}$. For instance, given a symmetric monoidal category $(\mathbb{C},\sigma)$, one has a strong monoidal functor $\mathbb{\Sigma}=(id_C,\sigma,id_I)\colon \mathbb{C}\to\mathbb{C}^t$ and thus an induced functor $\widetilde{\mathbb{\Sigma}}\colon \mathbf{Sem}(\mathbb{C})\to \mathbf{Sem}(\mathbb{C}^t)=\mathbf{Sem}(\mathbb{C})$. In fact $\widetilde{\mathbb{\Sigma}}\mathsf{S}=\mathsf{S}^{op}$. 


$\mathbf{MonCat}$ has a binary product $\mathbb{C}\times\mathbb{C}'$, with 
\begin{enumerate}
\item tensor $(C\times C')\times(C\times C')\xrightarrow{\Sigma_{2,3}}(C\times C)\times (C'\times C') \xrightarrow{\otimes\times \otimes'}C\times C'$, where $\Sigma_{2,3}((c,c'),(d,d'))=((c,d),(c',d'))$. 
\item unit object $(I,I')$,
\item and with obvious coherence constraints. 
\end{enumerate}
$\mathbf{MonCat}$ also has a terminal object namely the monoidal category $\mathbb{1}$ with underlying category $1$, the category with only one object and its identity arrow. One observes that $\mathbf{Sem}(\mathbb{C}\times\mathbb{C}')\simeq \mathbf{Sem}(\mathbb{C})\times\mathbf{Sem}(\mathbb{C}')$ canonically.

 The hom-functor $C(-,-)\colon {C}^{op}\times{C}\to \mathbf{Set}$, $C(b\xrightarrow{f}a,c\xrightarrow{h}d)\colon C(a,c)\to C(b,d)$, $g\mapsto h\circ g\circ f$, is a monoidal functor $\mathbb{Conv}=(C(-,-),\Gamma,\gamma)\colon \mathbb{C}^{op}\times\mathbb{C}\to\mathbb{Set}$, with $\Gamma_{(a,b),(c,d)}\colon C(a,b)\times C(c,d)\to C(a\otimes c,b\otimes d)$, $(f,g)\mapsto f\otimes g$, for $a,b,c,d\in \mathit{Ob} C$, and $\gamma\colon 1\to C(I,I)$, $0\mapsto id_I$.  Whence $\mathbb{Conv}$ induces   a functor $\mathsf{Conv}:=\widetilde{\mathbb{Conv}}\colon \mathbf{Cosem}(\mathbb{C})^{op}\times \mathbf{Sem}(\mathbb{C})\to \mathbf{Sem}$. Given  a semigroup $(S,S\otimes S\xrightarrow{\mu}S)$ and a cosemigroup  $(c,c\xrightarrow{\delta}c\otimes c)$ in $\mathbb{C}$, the the hom-set $C(c,S)$ becomes a semigroup, called the {\em convolution semigoup} $\mathsf{Conv}((c,\delta),(S,\mu))$ with multiplication $(g,f)\mapsto g\bullet_{{\mu,\delta}}f= \mu\circ (g\otimes f)\circ \delta$. Since $(I,\lambda_I^{-1}=\rho_{I}^{-1})$ is always a cosemigroup in any monoidal category, one has a functor $\mathsf{Conv}_I\colon \mathbf{Sem}(\mathbb{C})\to\mathbf{Sem}$ given by $\mathsf{Conv}_I((S,\mu)\xrightarrow{f}(S',\mu')):=(C(I,S),\bullet_\mu)\xrightarrow{C(id_I,f)}(C(I,S'),\bullet_{\mu'})$ with $\bullet_{\mu}:=\bullet_{\mu,\lambda_I^{-1}}$.

\section{Actions of a semigroup object}\label{sec:actions}

Let $\mathbb{C}:=(C,\otimes,I,\alpha,\lambda,\rho)$ be a monoidal category. Let $\mathsf{S}:=(S,\mu)$ be a semigroup in $\mathbb{C}$. Let $c\in \mathit{Ob} C$, and $\gamma\colon S\otimes c\to c$ (resp. $\rho\colon c\otimes S\to c$) be a $C$-morphism. $(S,\gamma)$ (resp. $(S,\rho)$) is called a {\em left $\mathsf{S}$-act} (resp. {\em right $\mathsf{S}$-act}) when the  diagram below on the left (resp. right) commutes.
\begin{equation}
\xymatrix@R=1em{
\eq[d]S\otimes(S\otimes c)\ar[r]^-{id_S\otimes \gamma} & S\otimes c\ar[d]^{\gamma}&&\eq[d](c\otimes S)\otimes S\ar[r]^-{\rho\otimes id_S} & c\otimes S\ar[d]^{\rho}\\
(S\otimes S)\otimes c\ar[d]_{\mu\otimes id_c}&c&&c\otimes(S\otimes S)\ar[d]_{id_c\otimes\mu}&c\\
S\otimes c\ar[ru]_{\gamma}&&&c\otimes S\ar[ru]_{\rho}&
}
\end{equation}
$\gamma$ (resp. $\rho$) then is referred to as the {\em left} (resp. {\em right})  {\em $\mathsf{S}$-action}. 

\begin{example}
By associativity, $(S,\mu)$  is both a left and a right $\mathsf{S}$-act. 
\end{example}

Given left (resp. right) $\mathsf{S}$-acts $(c,\gamma),(c',\gamma')$ (resp. $(c,\rho),(c',\rho')$) a $C$-morphism $c\xrightarrow{f}c'$ is said to be a {\em left} (resp. {\em right}) {\em $\mathsf{S}$-morphism} when the  diagram below on the left (resp. right) commutes.
\begin{equation}
\xymatrix@R=1em{
c\ar[r]^{f} & c'&&c\ar[r]^{f} & c'\\
\ar[u]^{\gamma}S\otimes c\ar[r]_{id_S\otimes f}&S\otimes c'\ar[u]_{\gamma'}&&\ar[u]^{\rho}c\otimes  S\ar[r]_{ f\otimes id_S}&c'\otimes S\ar[u]_{\rho'}
}
\end{equation}
This provides the categories ${}_{\mathsf{S}}\mathbf{Act}(\mathbb{C})$ and $\mathbf{Act}_{\mathsf{S}}(\mathbb{C})$ of left and right $\mathsf{S}$-acts respectively, together with obvious faithful forgetful functors ${}_{\mathsf{S}}|-|\colon {}_{\mathsf{S}}\mathbf{Act}(\mathbb{C})\to C$ and $|-|_{\mathsf{S}}\colon \mathbf{Act}_{\mathsf{S}}(\mathbb{C})\to C$. 

Let $\mathsf{S}=(S,\mu)$ and $\mathsf{S}'=(S',\mu')$ be semigroups in $\mathbb{C}$. Given an object $c$ together with a left $\mathsf{S}$-action $\gamma$ and a right $\mathsf{S}'$-action $\rho$, $(c,\gamma,\rho)$ is said to be  a {\em $\mathsf{S}\mhyphen \mathsf{S}'$-biact} when the following diagram commutes.
\begin{equation}
\xymatrix@R=1em{
\eq[d](S\otimes c)\otimes S' \ar[r]^-{\gamma\otimes id_{S'}}& c\otimes S'\ar[d]^{\rho}\\
S\otimes(c\otimes S') \ar[d]_{id_S\otimes\rho}&c\\
S\otimes c\ar[ru]_{\gamma}
}
\end{equation}
When $\mathsf{S}=\mathsf{S}$, a $\mathsf{S}\mhyphen\mathsf{S}$-biact is simply called a {\em two-sided $\mathsf{S}$-biact}. 

\begin{example}
$(S,\mu,\mu)$ is a $\mathsf{S}$-biact.
\end{example}

Given $\mathsf{S}\mhyphen \mathsf{S}'$-biacts  $(c,\gamma,\rho),(c',\gamma',\rho')$, a $C$-morphism $c\xrightarrow{f}c'$ which is both a left $\mathsf{S}$-morphism $(c,\gamma)\xrightarrow{f}(c',\gamma')$ and a right $\mathsf{S}'$-morphism $(c,\rho)\xrightarrow{f}(c',\rho')$ is referred to as a {\em $\mathsf{S}\mhyphen\mathsf{S}'$-morphism}, and a {\em two-sided $\mathsf{S}$-morphism} when $\mathsf{S}=\mathsf{S}'$. This provides the category ${}_{\mathsf{S}}\mathbf{Act}_{\mathsf{S}'}(\mathbb{C})$ together with a faithful forgetful functor ${}_{\mathsf{S}}|-|_{\mathsf{S}'}\colon {}_{\mathsf{S}}\mathbf{Act}_{\mathsf{S}'}(\mathbb{C})\to C$. 

One easily obtains the following result.
\begin{lemma}\label{lem:left_act_is_right_act_for_transpose}
Let $\mathsf{S}=(S,\mu)$ be a semigroup in $\mathbb{C}$. Then, $\mathbf{Act}_{\mathsf{S}}(\mathbb{C})={}_{\mathsf{S}}\mathbf{Act}(\mathbb{C}^t)$ and ${}_{\mathsf{S}}\mathbf{Act}(\mathbb{C})=\mathbf{Act}_{\mathsf{S}}(\mathbb{C}^t)$. 
\end{lemma}

\begin{example}
\begin{enumerate}
\item With $\mathsf{S}$ an ordinary semigroup, ${}_{\mathsf{S}}\mathbf{Act}(\mathbb{Set})$ is the usual category of left $\mathsf{S}$-acts. 

\item Let $R$ be a commutative ring with a unit. Let ${}_R\mathbb{Mod}$ be the monoidal category of left (unital) $R$-modules, with the usual tensor product $\otimes_R$ over $R$. Associative $R$-algebras and semigroup objects in ${}_R\mathbb{Mod}$ are essentially the same and so are left modules over an algebra and left acts over a semigroup object. 

\item Let $\mathbb{Ban}$ be the monoidal category with underlying category $\mathbf{Ban}$, the category of say complex Banach spaces with bounded linear maps, and with the projective tensor product $\hat{\otimes}$ (see~\cite{Ryan}). A semigroup $\mathsf{B}=(B,\mu)$ in $\mathbb{Ban}$ is essentially\footnote{Up to the change of the original norm by an equivalent sub-multiplicative norm~\cite{Poinsot}.} a non-unital Banach algebra $(B,*)$ with $x*y:=\mu(x\otimes y)$, $x,y\in B$. Left $\mathsf{B}$-acts are left Banach $(B,*)$-modules. 


\end{enumerate}
\end{example}


Let $(\mathbb{C},\sigma)$ be a symmetric monoidal category. Let $S,c,d$ be $C$-objects. Let $f\colon S\otimes c\to d$ (resp. $g\colon c\otimes S\to d$) be a $C$-morphism. One defines $f':= c\otimes S\xrightarrow{\sigma_{c,S}}S\otimes c\xrightarrow{f}d$ (resp. ${}^\backprime g\colon S\otimes c\xrightarrow{\sigma_{S,c}}c\otimes S\xrightarrow{g}c$). Of course, ${}^\backprime (f')=f$ and $({}^\backprime g)'=g$. 
\begin{proposition}\label{prop:com_mon}
Let $\mathsf{S}=(S,\mu)$ be a commutative semigroup in $\mathbb{C}$. There are functors $A_l\colon {}_{\mathsf{S}}\mathbf{Act}(\mathbb{C})\to \mathbf{Act}_{\mathsf{S}}(\mathbb{C})$, with $|-|_{\mathsf{S}}\circ A_l={}_{\mathsf{S}}|-|$ and $A_r\colon \mathbf{Act}_{\mathsf{S}}(\mathbb{C})\to {}_{\mathsf{S}}\mathbf{Act}(\mathbb{C})$ with ${}_{\mathsf{S}}|-|\circ A_r=|-|_{\mathsf{S}}$, which are inverse one from the other. There are full embedding functors $\mathbf{Act}_{\mathsf{S}}(\mathbb{C})\xrightarrow{B_l}{}_{\mathsf{S}}\mathbf{Act}_{\mathsf{S}}(\mathbb{C})\xleftarrow{B_r}{}_{\mathsf{S}}\mathbf{Act}(\mathbb{C})$ such that ${}_{\mathsf{S}}|-|_{\mathsf{S}}\circ B_l={}_{\mathsf{S}}|-|$ and ${}_{\mathsf{S}}|-|_{\mathsf{S}}\circ B_r=|-|_{\mathsf{S}}$, and the following diagram commutes.
\begin{equation}
\xymatrix@R=1em{
{}_{\mathsf{S}}\mathbf{Act}(\mathbb{C})\ar[rr]^{A_l}\ar[rd]_{B_l} && \mathbf{Act}_{\mathsf{S}}(\mathbb{C})\ar[ld]^{B_r}\\
&{}_{\mathsf{S}}\mathbf{Act}_{\mathsf{S}}(\mathbb{C})&
}
\end{equation}
\end{proposition}

\begin{proof}
 Let $(c,\gamma)$ be a left $\mathsf{S}$-act. $(c,\gamma')$ is a right $\mathsf{S}$-act since each cell of the following diagram commute, so commutes the surrounding diagram.
\begin{equation}
\xymatrix@R=1em@C=1em{
\eq[dd](c\otimes S)\otimes S\ar[rr]^{\sigma_{c,S}\otimes id_S} &&\eq[ld]\ar[rd]_{\sigma_{S\otimes c,S}}(S\otimes c)\otimes S\ar[rr]^{\gamma\otimes id_S}&&c\otimes S\ar[d]^{\sigma_{c,S}}\\
&S\otimes(c\otimes S)\ar[d]_{id_S\otimes\sigma_{c,S}} &&\eq[d]S\otimes(S\otimes c)\ar[r]_{id_S\otimes\gamma}&S\otimes c\ar[dd]^{\gamma}\\
\ar@/_3pc/[rrdd]_{id_c\otimes\mu}\ar[rd]_{\sigma_{c,S\otimes S}}c\otimes(S\otimes S)&\eq[d]S\otimes (S\otimes c) \eq[rrd]&&(S\otimes S)\otimes c\ar@/^2.5pc/[dd]^{\mu\otimes id_c}&&\\
&(S\otimes S)\otimes c\ar@{=}[rr]&&\ar[d]_{\mu\otimes id_c}(S\otimes S)\otimes c\ar[u]^{\sigma_{S,S}\otimes id_c}&c\\
&&c\otimes S\ar[r]_{\sigma_{c,S}}&S\otimes c\ar[ru]_{\gamma}&
}
\end{equation}
$(c,\gamma,\gamma')$ is a two-sided $\mathsf{S}$-act. Indeed all internal cells of the following diagram commute, so commutes the surrounding diagram. 
\begin{equation}
\xymatrix@R=1em{
\ar[ddd]_{id_S\otimes\sigma_{c,S}}S\otimes (c\otimes S)\eq[r]&(S\otimes c)\otimes S\ar[d]_{\sigma_{S\otimes c,S}}\ar[r]^{\gamma\otimes id_S} & c\otimes S\ar[rdd]^{\sigma_{c,S}} & \\
&S\otimes(S\otimes c)\eq[d] \ar[rrd]^{id_S\otimes\gamma}&&\\
&(S\otimes S)\otimes c\ar[r]_{\mu\otimes id_c}&S\otimes c\ar[rdd]_{\gamma}& S\otimes c\ar[dd]^{\gamma}\\
\ar[d]_{id_S\otimes\gamma}S\otimes(S\otimes c)\eq[r]&(S\otimes S)\otimes c\ar[u]^{\sigma_{S,S}\otimes id_c}\ar[ru]_{\mu\otimes id_c}&&\\
S\otimes c\ar[rrr]_{\gamma} &&&c
}
\end{equation}
Let $f\in {}_{\mathsf{S}}Act(\mathbb{C})((c_1,\gamma_1),(c_2,\gamma_2))$. Then, $f\in Act_{\mathsf{S}}(\mathbb{C})((c_1,\gamma_1'),(c_2,\gamma_2'))$, and thus $f\in {}_{\mathsf{S}}Act_{\mathsf{S}}(\mathbb{C})((c_1,\gamma_1,\gamma_1'),(c_2,\gamma_2,\gamma_2'))$. 
Indeed, the two cells of the following diagram commute so commutes the surrounding diagram.
\begin{equation}
\xymatrix@R=1em{
 c_1\ar[r]^{f}& c_2\\
\ar[u]^{\gamma_1}S\otimes c_1 \ar[r]_{id_S\otimes f}& S\otimes c_2\ar[u]_{\gamma_2}\\
\ar[u]^{\sigma_{c_1,S}}\ar[r]_{f\otimes id_S}c_1\otimes S & c_2\otimes S\ar[u]_{\sigma_{c_2,S}}
}
\end{equation}
Consequently one may define $A_l((c_1,\gamma_1)\xrightarrow{f}(c_2,\gamma_2)):=(c_1,\gamma_1')\xrightarrow{f}(c_2,\gamma_2')$ and $B_l((c_1,\gamma_1)\xrightarrow{f}(c_2,\gamma_2)):=(c_1,\gamma_1,\gamma_1')\xrightarrow{f}(c_2,\gamma_2,\gamma_2')$. 

That $A_r((c_1,\rho_1)\xrightarrow{f}(c_2,\rho_2)):=(c_1,{}^{\backprime}\rho_1)\xrightarrow{f}(c_2,{}^{\backprime}\rho_2)$ provides a functor from $Act_{\mathsf{S}}(\mathbb{C})$ to ${}_{\mathsf{S}}Act(\mathbb{C})$ and $B_r((c_1,\rho_1)\xrightarrow{f}(c_2,\rho_2)):=(c_1,{}^{\backprime}\rho_1,\rho_1)\xrightarrow{f}(c_2,{}^{\backprime}\rho_2,\rho_2)$ provides a functor from $Act_{\mathsf{S}}(\mathbb{C})$ to ${}_{\mathsf{S}}Act(\mathbb{C})_{\mathsf{S}}$ is obtained by symmetry. 
\end{proof}
\section{The multiplier monoid of a semigroup object}\label{sec:mult_mon}

In this section is introduced the multiplier monoid of a semigroup object together with the inner multipliers, induced by generalized elements. By dualization the comultiplier monoid of a cosemigroup object is freely provided. 

Throughout this section $\mathbb{C}=(C,\otimes,I)$ stands for a monoidal category. 
\subsection{The construction}

\begin{proposition}\label{prop:tens_prod_of_right_and_left_act}
Let $\mathsf{S}=(S,\mu)$ and $\mathsf{S}'=(S',\mu')$ be  semigroups in $\mathbb{C}$. 
 $\otimes\colon C\times C\to C$ lifts to   a functor ${}_{\mathsf{S}}\mathbf{Act}(\mathbb{C})\times \mathbf{Act}_{\mathsf{S}'}(\mathbb{C})\to {}_{\mathsf{S}}\mathbf{Act}_{\mathsf{S}'}(\mathbb{C})$, still denoted $\otimes$, that is, such that  the following diagram commutes. 
\begin{equation}
\xymatrix@R=1em{
\ar[d]_{{}_{\mathsf{S}}|-|\times |-|_{\mathsf{S}'}}{}_{\mathsf{S}}\mathbf{Act}(\mathbb{C})\times \mathbf{Act}_{\mathsf{S}'}(\mathbb{C})\ar[r]&{}_{\mathsf{S}}\mathbf{Act}_{\mathsf{S}'}(\mathbb{C})\ar[d]^{{}_{\mathsf{S}}|-|_{\mathsf{S}'}}\\
C\times C \ar[r]_{\otimes}& C
}
\end{equation}
\end{proposition}
\begin{proof}
Let $(c,\gamma)$ and $(c',\rho')$ be respectively a left  $\mathsf{S}$-act and a right $\mathsf{S}'$-act. Let $\gamma_{\rho'}:=S\otimes (c\otimes c')\simeq (S\otimes c)\otimes c'\xrightarrow{\gamma\otimes id_{c'}}c\otimes c'$, and let ${}_{\gamma}\rho':=(c\otimes c')\otimes S'\simeq c\otimes(c'\otimes S')\xrightarrow{id_c\otimes \rho'}c\otimes c'$. Then, one claims that $(c\otimes c',\gamma_{\rho'},{}_{\gamma}\rho')$ is a  two-sided $\mathsf{S}\mhyphen \mathsf{S}'$-biact. First, $(c\otimes c',\gamma_{\rho'})$ is a left $\mathsf{S}$-act because all the cells of the following diagram commute. 
\begin{equation}
\xymatrix{
\ar[d]_{\alpha^{-1}_{S,S,c\otimes c'}}S\otimes(S\otimes(c\otimes c'))\ar[r]^{id_S\otimes\alpha^{-1}_{S,c,c'}} &\ar[d]_{\alpha^{-1}_{S,S\otimes c,c'}}S\otimes((S\otimes c)\otimes c')\ar[r]^{id_S\otimes(\gamma\otimes id_{c'}}&S\otimes(c\otimes c')\ar[d]^{\alpha^{-1}_{S,c,c'}}\\
\ar[d]_{\mu\otimes(id_c\otimes id_{c'})}(S\otimes S)\otimes(c\otimes c')\ar[rd]_{\alpha^{-1}_{S\otimes S,c,c'}}&(S\otimes (S\otimes c))\otimes c'\ar[d]_{\alpha^{-1}_{S,S,c}\otimes id_{c'}}\ar[r]^{(id_S\otimes \gamma)\otimes id_{c'}}&(S\otimes c)\otimes c'\ar[d]^{\gamma\otimes id_{c'}}\\
\ar[rd]_{\alpha^{-1}_{S,c,c'}}S\otimes(c\otimes c')&((S\otimes S)\otimes c)\otimes c'\ar[d]_{(\mu\otimes id_c)\otimes id_{c'}}&c\otimes c'\\
&(S\otimes c)\otimes c'\ar[ru]_{\gamma\otimes id_{c'}}&
}
\end{equation}
Secondly, that $(c\otimes c',{}_{\gamma}\rho')$ is a right $\mathsf{S}'$-act is obtained similarly. Finally, $(c\otimes c',\gamma_{\rho'},{}_{\gamma}\rho')$ is a $\mathsf{S}\mhyphen \mathsf{S}'$-biact because all internal cells of the following diagrams commute, so does the surrounding diagram. 
\begin{equation}
\xymatrix@R=1.5em{
(S\otimes (c\otimes c'))\otimes S'\ar[d]_{\alpha_{S,c\otimes c',S'}}\ar[r]^{\alpha^{-1}_{S,c,c'}\otimes id_{S'}} & ((S\otimes c)\otimes c')\otimes S'\ar[r]^{(\gamma\otimes id_{c'})\otimes id_{S'}}\ar[d]_{\alpha_{S\otimes c,c',S'}}&(c\otimes c')\otimes S'\ar[d]^{\alpha_{c,c',S'}}\\
S\otimes ((c\otimes c')\otimes S')\ar[d]_{id_S\otimes\alpha_{c,c',S'}}&(S\otimes c)\otimes (c'\otimes S')\ar[r]^{\gamma\otimes (id_{c'}\otimes id_{S'})}\ar[dd]_{(id_S\otimes id_c)\otimes\rho'}&c\otimes (c'\otimes S')\ar[d]^{id_{c}\otimes\rho'}\\
S\otimes(c\otimes (c'\otimes S'))\ar[d]_{id_S\otimes(id_c\otimes\rho')}\ar[ru]^{\alpha^{-1}_{S,c,c'\otimes S'}}&&c\otimes c'\\
S\otimes(c\otimes c')\ar[r]_{\alpha^{-1}_{S,c,c'}}&(S\otimes c)\otimes c'\ar[ru]_{\gamma\otimes id_{c'}}&
}
\end{equation}

Let $f\in {}_{\mathsf{S}}\mathbf{Act}((c,\gamma_c),(d,\gamma_d))$ and let $f'\in \mathbf{Act}_{\mathsf{S}'}((c',\rho'_{c'}),(d',\rho'_{d'}))$. Then, $f\otimes f'\in {}_{\mathsf{S}}\mathbf{Act}_{\mathsf{S}'}((c\otimes c',(\gamma_c)_{\rho'_{c'}},{}_{\gamma_c}(\rho'_{c'})),(d\otimes d',(\gamma_d)_{\rho'_{d'}},{}_{\gamma_d}(\rho'_{d'})))$ because the internal cells of both following diagrams commute.
\begin{equation}
\xymatrix@R=1em{
c\otimes c' \ar[r]^{f\otimes f'}&d\otimes d'&&c\otimes c'\ar[r]^{f\otimes f'}&d\otimes d'\\
\ar[u]^{\gamma_c\otimes id_{c'}}(S\otimes c)\otimes c'\ar[r]_{(id_S\otimes f)\otimes f'}&\ar[u]_{\gamma_d\otimes id_{d'}}(S\otimes d)\otimes d'&&\ar[u]^{id_c\otimes \rho'_{c'}}c\otimes(c'\otimes S')\ar[r]_{f\otimes (f'\otimes id_{S'})}&d\otimes (d'\otimes S')\ar[u]_{id_d\otimes\rho'_{d'}}\\
\ar[u]^{\alpha^{-1}_{S,c,c'}}S\otimes (c\otimes c')\ar[r]_{id_S\otimes (f\otimes f')}&\ar[u]_{\alpha_{S,d,d'}}S\otimes(d\otimes d')&&\ar[u]^{\alpha_{c,c',S'}}\ar[r]_{(f\otimes f')\otimes id_{S'}}(c\otimes c')\otimes S'&(d\otimes d')\otimes S'\ar[u]_{\alpha_{d,d',S'}}
}
\end{equation}
Functoriality of the construction follows from that of $\otimes$.
\end{proof}

As a consequence of the associativity condition is the 
\begin{lemma}\label{lem:mu_is_bi_act_morphism}
Let $\mathsf{S}=(S,\mu)$ be a semigroup in $\mathbb{C}$. Then, $\mu\in {}_{\mathsf{S}}\mathbf{Act}_{\mathsf{S}}(\mathbb{C})((S,\mu)\otimes(S,\mu),(S,\mu,\mu))$.
\end{lemma}

Let $\mathsf{S}=(S,\mu)$ be a semigroup in $\mathbb{C}$.  
Let us introduce the set-theoretic maps $ \mathbf{Act}_{\mathsf{S}}(\mathbb{C})((S,\mu),(S,\mu))\xrightarrow{\mathsf{g}_{\mathbb{C},\mathsf{S}}} {}_{\mathsf{S}}\mathbf{Act}_{\mathsf{S}}(\mathbb{C})((S,\mu)\otimes(S,\mu),(S,\mu,\mu))$ using Prop.~\ref{prop:tens_prod_of_right_and_left_act} and Lemma~\ref{lem:mu_is_bi_act_morphism} by $\mathsf{g}_{\mathbb{C},\mathsf{S}}(L):=\mu\circ (id_S\otimes L)$ and $ {}_{\mathsf{S}}\mathbf{Act}((S,\mu),(S,\mu))\xrightarrow{\mathsf{d}_{\mathbb{C},\mathsf{S}}} {}_{\mathsf{S}}\mathbf{Act}_{\mathsf{S}}((S,\mu)\otimes(S,\mu),(S,\mu,\mu))$ by $\mathsf{d}_{\mathbb{C},\mathsf{S}}(R):=\mu\circ (R\otimes id_S)$. Let $$\mathit{Mult}_{\mathbb{C}}(\mathsf{S}):=\left(\mathbf{Act}_{\mathsf{S}}(\mathbb{C})((S,\mu),(S,\mu))\right){}_{\mathsf{g}}\times_{\mathsf{d}}\left({}_{\mathsf{S}}\mathbf{Act}(\mathbb{C})((S,\mu),(S,\mu))\right)$$ (fibre product computed in $\mathbf{Set}$). The members of $\mathit{Mult}_{\mathbb{C}}(\mathsf{S})$ are referred to as the {\em multipliers} of $\mathsf{S}$.

\begin{proposition}\label{prop:monoid_of_multipliers_in_mon_cat}
Let $\mathsf{S}=(S,\mu)$ be a semigroup in $\mathbb{C}$.   $(\mathit{Mult}_{\mathbb{C}}(\mathsf{S}),\star,(id_S,id_S))$ is a monoid  denoted by $\mathsf{Mult}_{\mathbb{C}}(\mathsf{S})$, under $(L',R')\star(L,R):=(L'\circ L,R\circ R')$. In fact, $\mathit{Mult}_{\mathbb{C}}(\mathsf{S})$ is a submonoid of the product monoid $\mathbf{Act}_{\mathsf{S}}(\mathbb{C})((S,\mu),(S,\mu))\times {}_{\mathsf{S}}\mathbf{Act}((S,\mu),(S,\mu))^{op}$. 
\end{proposition}

In what follows $\mathsf{Mult}_{\mathbb{C}}(\mathsf{S})$ is referred to as the {\em multiplier monoid} of the semigroup object $\mathsf{S}$ in $\mathbb{C}$. The following result is clear. 

\begin{lemma}\label{lem:formal_symmetry}
Let $\mathbb{C}$ be a monoidal category and let $\mathsf{S}=(S,\mu)$ be a semigroup object in $\mathbb{C}$. Then, $\mathit{Mult}_{\mathbb{C}^t}(\mathsf{S})=\{\, (R,L)\colon (L,R)\in \mathit{Mult}_{\mathbb{C}}(\mathsf{S})\,\}$ and $(L,R)\mapsto (R,L)$ provides an isomorphism of monoids from $\mathsf{Mult}_{\mathbb{C}}(\mathsf{S})$ to $\mathsf{Mult}_{\mathbb{C}^t}(\mathsf{S})^{op}$. 
\end{lemma}

\subsection{Inner multipliers induced by generalized elements}\label{sec:inner_mult_and_gen_elt}

Since they might be no elements for an object $S$ in an arbitrary monoidal category $\mathbb{C}$, one cannot  directly define the inner multipliers as in the set-theoretic situation. To get rid of this obstruction one uses instead the {\em generalized elements}, that is, the members of the hom-set $C(I,S)$. 
\begin{lemma}
Let $\mathsf{S}=(S,\mu)$ be a semigroup in $\mathbb{C}$. Let $f\in C(I,S)$. Let $L_f:=S\xrightarrow{\lambda_S^{-1}}I\otimes S\xrightarrow{f\otimes id_S}S\otimes S\xrightarrow{\mu}S$. Then, $L_f\in \mathbf{Act}_{\mathsf{S}}(\mathbb{C})((S,\mu),(S,\mu))$. 
\end{lemma}
\begin{proof}
All the cells commute in the diagram below, so commutes the surrounding diagram. (The left down triangle commutes by coherence.)
\begin{equation}
\xymatrix@R=1em{
S\ar[r]^{\lambda^{-1}_S} & I\otimes S\ar[r]^{f\otimes id_S} & S\otimes S\ar[r]^{\mu} & S\\
&I\otimes (S\otimes S) \ar[ld]_{\lambda_{S\otimes S}^{-1}}\ar[u]^{id_I\otimes \mu}\ar[r]^{f\otimes (id_S\otimes id_S)}&\ar[u]_{id_S\otimes \mu} S\otimes (S\otimes S) & \\
S\otimes S \ar[uu]^{\mu}\ar[r]_{\lambda_{S}^{-1}\otimes id_S}&(I\otimes S)\otimes S \ar[r]_{(f\otimes id_S)\otimes id_S}\eq[u]& (S\otimes S)\otimes S\eq[u]\ar[r]_{\mu} & S\otimes S\ar[uu]_{\mu}
}
\end{equation}
\end{proof}

 By symmetry,
\begin{lemma}
Let $\mathsf{S}=(S,\mu)$ be a semigroup in $\mathbb{C}$. Let $f\in C(I,S)$. Let $R_f:=S\xrightarrow{\rho_S^{-1}}S\otimes I\xrightarrow{id_S\otimes f}S\otimes S\xrightarrow{\mu}S$. Then, $R_f\in {}_{\mathsf{S}}\mathbf{Act}(\mathbb{C})((S,\mu),(S,\mu))$. 
\end{lemma}

\begin{lemma}\label{lem:inner_multiplier}
Let $\mathsf{S}=(S,\mu)$ be a semigroup in $\mathbb{C}$. Let $f\in C(I,S)$. Then, $(L_f,R_f)\in \mathit{Mult}_{\mathbb{C}}(\mathsf{S})$.
\end{lemma}
\begin{proof}
\begin{equation}
\begin{array}{lll}
\mu\circ (R_f\otimes id_S)&=&\mu\circ (\mu\otimes id_S)\circ ((id_S\otimes f)\otimes id_S)\circ (\rho_S^{-1}\otimes id_S)\\
&=&\mu\circ (id_S\otimes \mu)\circ \alpha_{S,S,S}\circ ((id_S\otimes f)\otimes id_S)\circ (\rho_S^{-1}\otimes id_S)\\
&&\mbox{(by associativity of $\mu$)}\\
&=&\mu\circ (id_S\otimes\mu)\circ (id_S\otimes (f\otimes id_S))\circ \alpha_{S,I,S}\circ  (\rho_S^{-1}\otimes id_S)\\
&&\mbox{(by naturality of $\alpha$)}\\
&=&\mu\circ (id_S\otimes\mu)\circ (id_S\otimes (f\otimes id_S))\circ(id_S\otimes \lambda_S^{-1})\\
&&\mbox{(by coherence)}\\
&=&\mu\circ(id_S\otimes L_f).
\end{array}
\end{equation}
\end{proof}

According to Lemma~\ref{lem:inner_multiplier}, for each semigroup $\mathsf{S}$ in $\mathbb{C}$, one has a map $C(I,S)\xrightarrow{{\mathcal{M}_{\mathbb{C},\mathsf{S}}}}\mathit{Mult}_{\mathbb{C}}(\mathsf{S})$ given by $f\mapsto (L_f,R_f)$. The image of this map is the set of all {\em inner multipliers} of $\mathsf{S}$ in $\mathbb{C}$ and is denoted $InnMult_{\mathbb{C}}(\mathsf{S}):=\mathcal{M}_{\mathsf{S}}(C(I,S))$.  Moreover $L_f$ and $R_f$ are respectively called the {\em inner left translation} and the {\em inner right translation} induced by $f$. $\mathcal{M}_{\mathbb{C},\mathsf{S}}$ plays the role of the canonical homomorphism $\mathfrak{M}_{\mathsf{S}}$ for an  ordinary semigroup $\mathsf{S}$.


\begin{proposition}\label{prop:hom_M}
Let $\mathsf{S}=(S,\mu)$ be a semigroup in $\mathbb{C}$.  Then, $\mathcal{M}_{\mathsf{S}}\colon \mathsf{Conv}_I(\mathsf{S})\to |\mathsf{Mult}_{\mathbb{C}}(\mathsf{S})|=(\mathit{Mult}_{\mathbb{C}}(\mathsf{S}),\star)$ is a homomorphism of semigroups and in particular, $InnMult_{\mathbb{C}}(\mathsf{S})$ is a sub-semigroup of $(\mathit{Mult}_{\mathbb{C}}(\mathsf{S}),\star)$. 
\end{proposition}

\begin{proof}
Let $f,g\in C(I,S)$. One has 
\begin{equation}
\begin{array}{l}
L_{g\bullet_\mu f}=\mu\circ (\mu\otimes id_S)\circ ((g\otimes f)\otimes id_S)\circ (\lambda_I^{-1}\otimes id_S)\circ \lambda_S^{-1}\\
=\mu \circ (id_S\otimes \mu)\circ \alpha_{S,S,S}\circ ((g\otimes f)\otimes id_S)\circ (\lambda_I^{-1}\otimes id_S)\circ \lambda_S^{-1}\\
\mbox{(by associativity)}\\
=\mu\circ (id_S\otimes \mu)\circ\alpha_{S,S,S}\circ ((g\otimes id_S)\otimes id_S)\circ ((id_I\otimes f)\otimes id_S)\circ (\lambda_I^{-1}\otimes id_S)\circ \lambda_S^{-1}\\
=\mu\circ (id_S\otimes\mu)\circ (g\otimes (id_S\otimes id_S))\circ \alpha_{I,S,S}\circ  ((id_I\otimes f)\otimes id_S)\circ (\lambda_I^{-1}\otimes id_S)\circ \lambda_S^{-1}\\
\mbox{(by naturality of $\alpha$)}\\
=\mu\circ (g\otimes id_S)\circ (id_I\otimes \mu)\circ \alpha_{I,S,S}\circ  ((id_I\otimes f)\otimes id_S)\circ (\lambda_I^{-1}\otimes id_S)\circ \lambda_S^{-1}\\
=\mu\circ (g\otimes id_S)\circ (id_I\otimes \mu)\circ (id_\otimes (f\otimes id_S))\circ \alpha_{I,I,S}\circ  (\lambda_I^{-1}\otimes id_S)\circ \lambda_S^{-1}\\
\mbox{(by naturality of $\alpha$)}\\
=\mu\circ (g\otimes id_S)\circ (id_I\otimes \mu)\circ (id_\otimes (f\otimes id_S))\circ(id_I\otimes \lambda^{-1}_S)\circ\lambda_{S}^{-1}\\
\mbox{(by coherence)}\\
=\mu\circ (g\otimes id_S)\circ (id_I\otimes \mu)\circ \lambda_{S\otimes S}^{-1}\circ (f\otimes id_S)\circ \lambda_{S}^{-1}\\
=\mu\circ (g\otimes id_S)\circ \lambda_{S}^{-1}\circ \mu\circ (f\otimes id_S)\circ \lambda_{S}^{-1}\\
\mbox{(by naturality of $\lambda$ for the two above equations)}\\
=L_g\circ L_f.
\end{array}
\end{equation}
By symmetry one also has $R_{g\bullet_\mu f}=R_f\circ R_g$. 
\end{proof}

\begin{remark}
The proof of Prop.~\ref{prop:hom_M} also shows that the maps $ C(I,S)\xrightarrow{\mathcal{L}_{\mathbb{C},\mathsf{S}}} \mathbf{Act}_{\mathsf{S}}(\mathbb{C})((S,\mu),(S,\mu))$, $f\mapsto L_f$, and $C(I,S)\xrightarrow{\mathcal{R}_{\mathbb{C},\mathsf{S}}} {}_{\mathsf{S}}\mathbf{Act}(\mathbb{C})((S,\mu),(S,\mu))$, $f\mapsto R_f$, are   homomorphisms $\mathsf{Conv}_I(\mathsf{S})\to (\mathbf{Act}_{\mathsf{S}}(\mathbb{C})((S,\mu),(S,\mu)),\circ)$ and $\mathsf{Conv}_I(\mathsf{S})^{op}\to ({}_{\mathsf{S}}\mathbf{Act}(\mathbb{C})((S,\mu),(S,\mu)),\circ)$ respectively. 
\end{remark}

\begin{example}\label{ex:generalized_elements}
 Let $\mathsf{S}:=(S,*)$  be a semigroup. Then, $\mathsf{Mult}_{\mathbb{Set}}(\mathsf{S})$ is the ordinary translational hull $\mathsf{TrHull}(\mathsf{S})$ of $\mathsf{S}$. Moreover, $\mathsf{S}\simeq \mathsf{Conv}_1(\mathsf{S})$ under $x\mapsto \left (\mathit{gen}_{\mathsf{S}}(x)\colon 1\to S\right)$ with $\mathit{gen}_{\mathsf{S}}(x)(0):=x$. (When $S=\emptyset$, then $\mathit{gen}_{\mathsf{S}}$ is the empty map.) Moreover $\mathfrak{M}_{\mathsf{S}}=\mathcal{M}_{\mathbb{Set},\mathsf{S}}\circ \mathit{gen}_{\mathsf{S}}$.
\end{example}

%
%

\begin{example}
Let $\mathsf{S}$ be a semigroup object in $\mathbb{Ban}$, considered as a Banach algebra. Then, $\mathsf{Mult}_{\mathbb{Ban}}(\mathsf{S})$ is the underlying semigroup of the multiplier Banach algebra of $\mathsf{S}$ (see e.g.~\cite[Theorem~1.2.4, p.~29]{Palmer}).
\end{example}

Let $\mathsf{S}=(S,\mu)$ be a semigroup object in $\mathbb{C}$. Let $f,g\in C(I,S)$. Then, 
\begin{equation}\label{eq:computations_for_concretization_homomorphism}
\begin{array}{lll}
L_f\circ g&=&\mu\circ (f\otimes id_S)\circ \lambda_S^{-1}\circ g\\
&=&\mu\circ (f\otimes id_S)\circ (id_I\otimes g)\circ \lambda_I^{-1}\\
&=&f\bullet_\mu g
\end{array}
\end{equation}
and almost identically, \begin{equation}\label{eq:computations_for_concretization_homomorphism_2}
R_f\circ g=g\bullet_\mu f.
\end{equation} From these computations one deduces the following
\begin{lemma}\label{lem:injectivity_of_can_map_for_left_and_right_non_degenerate_sem}
Let $\mathsf{S}=(S,\mu)$ be a semigroup object in $\mathbb{C}$. Assume that  $\mathsf{Conv}_I(\mathsf{S})$ is non-degenerate (see  Definition~\ref{def:non-degeneracy_of_a_semigroup}). Then, $\mathcal{M}_{\mathbb{C},\mathsf{S}}\colon \mathsf{Conv}_{I}(\mathsf{S})\to (\mathit{Mult}_{\mathbb{C}}(\mathsf{S}),\star)$  is one-to-one. 
\end{lemma}

\begin{proposition}\label{prop:for_com_sem_left_equal_right}
Let us assume that $(\mathbb{C},\sigma)$ is a symmetric monoidal category. Let $\mathsf{S}$ be a commutative semigroup in $\mathsf{S}$. Then, one has $\{\, (L,L)\colon L\in {}_{\mathsf{S}}\mathit{Act}(\mathbb{C})((S,\mu),(S,\mu))\,\}=\{\, (R,R)\colon R\in \mathit{Act}_{\mathsf{S}}(\mathbb{C})((S,\mu),(S,\mu))\,\}\subseteq \mathit{Mult}_{\mathbb{C}}(\mathsf{S})$. When furthermore $\mu$ is an epimorphism, then $\mathit{Mult}_{\mathbb{C}}(\mathsf{S})=\{\, (L,L)\colon L\in {}_{\mathsf{S}}\mathbf{Act}(\mathbb{C})((S,\mu),(S,\mu))\,\}$. In this  case, the multiplier monoid $\mathsf{Mult}_{\mathbb{C}}(\mathsf{S})$  and the endomorphism monoid $({}_{\mathsf{S}}\mathbf{Act}(\mathbb{C})((S,\mu),(S,\mu)),\circ,id_S)$ are isomorphic under $(L,L)\mapsto L$. 
\end{proposition}

\begin{proof}
The first equality is a consequence of Prop.~\ref{prop:com_mon} since one may notice that the isomorphism $A_l\colon {}_{\mathsf{S}}\mathbf{Act}(\mathbb{C})\to \mathbf{Act}_{\mathsf{S}}(\mathbb{C})$ restricts to the identity from ${}_{\mathsf{S}}\mathbf{Act}(\mathbb{C})((S,\mu),(S,\mu))$ to $\mathbf{Act}_{\mathsf{S}}(\mathbb{C})((S,\mu),(S,\mu))$ (because $A_l(S,\mu)=(S,\mu\circ \sigma_{S,S})=(S,\mu)$).  Let $L\in  {}_{\mathsf{S}}\mathbf{Act}(\mathbb{C})((S,\mu),(S,\mu))$. Then, $\mu\circ (id_S\otimes L)=\mu\circ \sigma_{S,S}\circ(L\otimes id_S)=\mu\circ (L\otimes id_S)$. Since $L\in \mathbf{Act}_{\mathsf{S}}((S,\mu),(S,\mu))$ it follows that $(L,L)\in \mathit{Mult}_{\mathbb{C}}(\mathsf{S})$. Let us assume that $\mu$ is an epimorphism. Let $(L,R)\in \mathit{Mult}_{\mathbb{C}}(\mathsf{S})$. Then, $R\circ \mu=\mu\circ (id_S\otimes R)=\mu\circ\sigma_{S,S}\circ (R\otimes id_S)=\mu\circ (R\otimes id_S)=\mu\circ (id_S\otimes L)=\mu\circ\sigma_{S,S}\circ (L\otimes id_S)=\mu\circ (L\otimes id_S)=L\circ\mu$, and since $\mu$ is an epimorphism, $R=L$. 
\end{proof}

\subsection{The case of monoids}

A {\em monoid object}, or simply a {\em monoid} in a monoidal category $\mathbb{C}$ is a triple $\mathsf{M}=(M,\mu,\eta)$ where $(M,\mu)$ is a semigroup in $\mathbb{C}$ and $\eta\colon I\to M$ is a $C$-morphism such that the diagrams
\begin{equation}
\xymatrix@R=1em{
I\otimes M \ar[r]^{\eta\otimes id_M}\ar[rd]_{\lambda_M} &M\otimes M\ar[d]^{\mu} & \ar[l]_{id_M\otimes\eta}\ar[ld]^{\rho_M}M\otimes I\\
&M&
}
\end{equation}
commute. When $\mathsf{M}$ is a monoid, the convolution semigroup $\mathsf{Conv}_{I}(M,\mu)$ becomes an ordinary monoid $\mathsf{Conv}_I(\mathsf{M}):=(C(I,M),\bullet_\mu,\eta)$. 

\begin{remark}\label{lem:generalized_elements_for_ordinary_monoids}
An  ordinary monoid $\mathsf{M}=(M,*,1_M)$ is a monoid object in $\mathbb{Set}$ with $\eta:=\mathit{gen}_{|\mathsf{M}|}(1_M)$  (see Example~\ref{ex:generalized_elements}).  Besides $\mathit{gen}_{|\mathsf{M}|}\colon |\mathsf{M}|\simeq \mathsf{Conv}_1(M,*)$ lifts to an isomorphism $\mathit{gen}_{\mathsf{M}}\colon \mathsf{M}\simeq \mathsf{Conv}_1(\mathsf{M})$ of monoids. \end{remark}

Similarly to the set-theoretic setting (Lemma~\ref{lem:can_hom_is_iso_for_mon}) one has the
\begin{proposition}\label{prop:mult_alg_of_a_monoid}
Let $\mathsf{M}=(M,\mu,\eta)$ be a monoid in $\mathbb{C}$. Then, $\mathcal{M}_{\mathbb{C},(M,\mu)}$ is the underlying homomorphism of semigroups of an isomorphism of monoids  $\mathcal{M}_{\mathbb{C},\mathsf{M}}\colon \mathsf{Conv}_I(\mathsf{M}) \simeq \mathsf{Mult}_{\mathbb{C}}(|\mathsf{M}|)$. In particular a monoid object has only inner multipliers.
\end{proposition}
\begin{proof}
Let $L\in \mathbf{Act}_{(M,\mu)}(\mathbb{C})((M,\mu),(M,\mu))$. Then, $L=L_{L\circ \eta}$. Indeed, 
\begin{equation}
\begin{array}{l}
M\xrightarrow{\lambda_M^{-1}}I\otimes M\xrightarrow{\eta\otimes id_M}M\otimes M\xrightarrow{L\otimes id_M}M\otimes M\xrightarrow{\mu}M\\
=M\xrightarrow{\lambda_M^{-1}}I\otimes M\xrightarrow{\eta\otimes id_M}M\otimes M\xrightarrow{\mu}M\xrightarrow{L}M\\
\mbox{(since $L$ is a right $(M,\mu)$-morphism)}\\
=M\xrightarrow{id_M}M\xrightarrow{L}M.\\
\mbox{(since $\eta$ is the unit of $\mathsf{M}$)}
\end{array}
\end{equation}
By symmetry for $R\in {}_{(M,\mu)}\mathbf{Act}(\mathbb{C})((M,\mu),(M,\mu))$, $R=R_{R\circ \eta}$. 
Now assume that $(L,R)\in \mathit{Mult}_{\mathbb{C}}(M,\mu)$.  Let us check that $R\circ \eta=L\circ \eta$. 
For each $f\in C(M,M)$, one has 
\begin{equation}
\begin{array}{lll}
\mu\circ ((f\otimes\eta)\otimes\eta)&=&\mu\circ (id_M\otimes \eta)\circ ((f\circ\eta)\otimes id_I)\\
&=&\rho_M\circ ((f\circ\eta)\otimes id_I)\\
&=&(f\circ \eta)\circ \rho_I\\
&&\mbox{(by naturality of $\rho$)}
\end{array}
\end{equation}
and 
\begin{equation}
\begin{array}{lll}
\mu\circ (\eta\otimes(f\circ \eta))&=&\mu\circ ( \eta\otimes id_M)\circ (id_I\otimes (f\circ\eta))\\
&=&\lambda_M\circ (id_I\otimes (f\circ\eta))\\
&=&(f\circ \eta)\circ \lambda_I.\\
&&\mbox{(by naturality of $\lambda$)}
\end{array}
\end{equation}
Furthermore, 
\begin{equation}
\begin{array}{lll}
\mu\circ ((R\circ\eta)\otimes\eta)&=&\mu\circ (R\otimes id_M)\circ (\eta\otimes\eta)\\
&=&\mu\circ (id_M\otimes L)\circ(\eta\otimes \eta)\\
&=&\mu\circ (\eta\otimes id_M)\circ (id_I\otimes (L\circ\eta))\\
&=&\lambda_M\circ (id_I\otimes (L\circ \eta))\\
&=&L\circ\eta\circ \lambda_I
\end{array}
\end{equation}
Therefore by the above $(R\circ \eta)\circ \rho_I=L\circ\eta\circ \lambda_I$ so that $R\circ\eta=L\circ\eta\circ \lambda_I\circ \rho_I^{-1}=L\circ \eta$. So with$f:=L\circ\eta=R\circ \eta$, one thus has $\mathcal{M}_{\mathbb{C},(M,\mu)}(f):=(L,R)$ which shows that $\mathcal{M}_{\mathbb{C},(M,\mu)}\colon C(I,M)\to \mathit{Mult}_{\mathbb{C}}(M,\mu)$ is onto. 

Let $f,g\in C(I,M)$ so that $L_f=L_g$. Then, because $\eta$ is the unit of $\mathsf{Conv}_I(\mathsf{M})$,  $f=\mu\circ(f\otimes id_M)\circ \lambda_M^{-1}\circ\eta=L_f\circ \eta=L_g\circ \eta=\mu\circ(g\otimes id_M)\circ \lambda_M^{-1}\circ\eta=g$. Assume that $R_f=R_g$. Then, $f=\mu\circ (id_M\otimes f)\circ \rho_M^{-1}\circ\eta=R_f\circ \eta=R_g\circ\eta=\mu\circ (id_M\otimes g)\circ \rho_M^{-1}\circ \eta=g$. 

As a consequence $\mathcal{M}_{\mathbb{C},(M,\mu)}$ is also one-to-one, and hence provides an isomorphism of semigroups $\mathsf{Conv}_I(M,\mu)\to (\mathit{Mult}(M,\mu),\star)$. As, $\mathcal{M}_{\mathbb{C},(M,\mu)}(\eta)=(L_\eta,R_{\eta})=(id_M,id_M)$  since $L_\eta=\mu\circ (\eta\otimes id_M)\circ \lambda_M^{-1}=\lambda_M\circ\lambda_M^{-1}=id_M$ and $R_\eta=\mu\circ (id_M\otimes \eta)\circ\rho_M^{-1}=\rho_M\circ \rho_M^{-1}=id_M$, $\mathcal{M}_{\mathbb{C},(M,\mu)}$ is an isomorphism of monoids from $\mathsf{Conv}_I(\mathsf{M})$ to $\mathsf{Mult}_{\mathbb{C}}(M,\mu)$. 
\end{proof}

\subsection{Comultipliers by dualization}\label{sec:comult}

By a simple dualization process, that is, essentially by replacing a category by its opposite, one now summarizes without proofs some important results about coactions of cosemigroups, deduced from that of actions of semigroups and one introduces the notion of comultipliers. 

Let $\mathsf{S}:=(S,\delta)$ be a cosemigroup in $\mathbb{C}$, that is, a semigroup in $\mathbb{C}^{op}$. A {\em left} (resp. {\em right}) {\em $\mathsf{S}$-coact} $(c,\beta)$ (resp. $(c,\tau)$) is an object of ${}_{\mathsf{S}}Act(\mathbb{C}^{op})$ (resp. $Act_{\mathsf{S}}(\mathbb{C}^{op})$). In what follows $\beta$ (resp. $\tau$) is referred to as a {\em left} (resp. {\em right}) {\em $\mathsf{S}$-coaction}. 

\begin{example}
By coassociativity, $(C,\delta)$ is both a left and a right $\mathsf{S}$-coact. 
\end{example}

Now on defines the category ${}_{\mathsf{S}}\mathbf{Coact}(\mathbb{C}):={}_{\mathsf{S}}\mathbf{Act}(\mathbb{C}^{op})^{op}$ (respectively,  $\mathbf{Coact}_{\mathsf{S}}(\mathbb{C}):=\mathbf{Act}_{\mathsf{S}}(\mathbb{C}^{op})^{op}$) of left (resp. right) $\mathsf{S}$-coacts.

Let $\mathsf{S}=(S,\delta)$ and $\mathsf{S}'=(S',\delta')$ be cosemigroups in $\mathbb{C}$. A {\em $\mathsf{S}\mhyphen\mathsf{S}'$-bi-coact} is a $\mathsf{S}\mhyphen\mathsf{S}'$-biact in $\mathbb{C}^{op}$, and one defines ${}_{\mathsf{S}}\mathbf{Coact}_{\mathsf{S}'}(\mathbb{C}):={}_{\mathsf{S}}\mathbf{Act}_{\mathsf{S}'}(\mathbb{C}^{op})^{op}$. When $\mathsf{S}=\mathsf{S}'$, a $\mathsf{S}\mhyphen\mathsf{S}'$-bi-coact is simply called a {\em (two-sided) $\mathsf{S}$-bi-coact}. 

\begin{example}
$(S,\delta,\delta)$ is a $\mathsf{S}$-bi-coact.
\end{example}

\begin{lemma}
Let $\mathsf{S}=(S,\delta)$ be a cosemigroup in $\mathbb{C}$. Then, $\mathbf{Coact}_{\mathsf{S}}(\mathbb{C})={}_{\mathsf{S}}\mathbf{Coact}(\mathbb{C}^{t})$.
\end{lemma}

%

\begin{proposition}\label{prop:tens_prod_of_right_and_left_coact}
Let $\mathsf{S}=(S,\delta)$ and $\mathsf{S}'=(S',\delta')$ be  cosemigroups in a monoidal category $\mathbb{C}$. 
 $\otimes\colon C\times C\to C$ lifts to   a functor ${}_{\mathsf{S}}\mathbf{Coact}(\mathbb{C})\times \mathbf{Coact}_{\mathsf{S}'}(\mathbb{C})\to {}_{\mathsf{S}}\mathbf{Coact}_{\mathsf{S}'}(\mathbb{C})$, that is, such that  the following diagram commutes. 
\begin{equation}
\xymatrix{
\ar[d]_{{}_{\mathsf{S}}|-|^{op}\times|-|_{\mathsf{S}'}^{op}}{}_{\mathsf{S}}\mathbf{Coact}(\mathbb{C})\times \mathbf{Coact}_{\mathsf{S}'}(\mathbb{C})\ar[r]&{}_{\mathsf{S}}\mathbf{Coact}_{\mathsf{S}'}(\mathbb{C})\ar[d]^{{}_{\mathsf{S}}|-|_{\mathsf{S}'}^{op}}\\
C\times C \ar[r]_{\otimes}& C
}
\end{equation}

\end{proposition}

\begin{lemma}\label{lem:mu_is_bi_coact_morphism}
$\delta\in {}_{\mathsf{S}}\mathbf{Coact}_{\mathsf{S}}(\mathbb{C})((S,\delta),(S,\delta)\otimes (S,\delta))$ for a cosemigroup $\mathsf{S}=(S,\delta)$ in $\mathbb{C}$.
\end{lemma}

Let $\mathsf{S}=(S,\delta)$ be a cosemigroup in $\mathbb{C}$.  Define $\mathit{Comult}_{\mathbb{C}}(\mathsf{S}):=\mathit{Mult}_{\mathbb{C}^{op}}(\mathsf{S})$. In other words, $$\mathit{Comult}_{\mathbb{C}}(\mathsf{S}):=\mathbf{Coact}_{\mathsf{S}}(\mathbb{C})((S,\delta),(S,\delta)){}_{\mathsf{g}_{\mathbb{C}^{op},\mathsf{S}}}\times_{\mathsf{d}_{\mathbb{C}^{op},\mathsf{S}}}{}_{\mathsf{S}}\mathbf{Coact}(\mathbb{C})((S,\delta),(S,\delta))$$
with $\mathsf{g}_{\mathbb{C}^{op},\mathsf{S}}(L)=\mathsf{g}(L):=(id_S\otimes L)\circ \delta$ and $\mathsf{d}_{\mathbb{C}^{op},\mathsf{S}}(R)= (R\otimes id_S)\circ \delta$. The members of $\mathit{Comult}_{\mathbb{C}}(\mathsf{S})$ are referred to as the {\em comultipliers} of $\mathsf{S}$. 

\begin{proposition}\label{prop:monoid_of_comultipliers_in_mon_cat}
Let $\mathsf{S}=(S,\delta)$ be a cosemigroup in $\mathbb{C}$.  Then, $\mathsf{Comult}_{\mathbb{C}}(\mathsf{S}):=(\mathit{Comult}_{\mathbb{C}}(\mathsf{S}),\star,(id_S,id_S))$ is a monoid under $(L',R')\star(L,R):=(L\circ L',R'\circ R)$. $\mathit{Comult}_{\mathbb{C}}(\mathsf{S})$ is in fact a submonoid of  $\mathbf{Coact}_{\mathsf{S}}(\mathbb{C})((S,\delta),(S,\delta))^{op}\times {}_{\mathsf{S}}\mathbf{Coact}(\mathbb{C})((S,\delta),(S,\delta))$ (product monoid). 
\end{proposition}

In what follows $\mathsf{Comult}_{\mathbb{C}}(\mathsf{S})$ is referred to as the {\em comultiplier monoid} of $\mathsf{S}$ in $\mathbb{C}$.

\begin{lemma}
Let $\mathsf{C}=(S,\delta)$ be a cosemigroup in $\mathbb{C}$. Let $f\in C(S,I)$. 
\begin{enumerate}
\item $L_f:=S\xrightarrow{\delta}S\otimes S\xrightarrow{f\otimes id_S}I\otimes S\xrightarrow{\lambda_S}S$. Then, $L_f\in \mathbf{Coact}_{\mathsf{S}}(\mathbb{C})((S,\delta),(S,\delta))$. 
\item  $R_f:=S\xrightarrow{\delta}S\otimes S\xrightarrow{id_S\otimes f}S\otimes I\xrightarrow{\rho_S}S$. Then, $R_f\in {}_{\mathsf{S}}\mathbf{Coact}(\mathbb{C})((S,\delta),(S,\delta))$. 
\item Moreover $(L_f,R_f)\in \mathit{Comult}_{\mathbb{C}}(\mathsf{S})$.
\end{enumerate}
\end{lemma}

Let $\mathsf{S}=(S,\delta)$ be a cosemigroup in $\mathbb{C}$. Let $\mathcal{C}_{\mathbb{C},\mathsf{S}}:=\mathcal{M}_{\mathbb{C}^{op},\mathsf{S}}\colon C(S,I)\to \mathit{Comult}_{\mathbb{C}}(\mathsf{S})$ be given by $\mathcal{C}_{\mathbb{C},\mathsf{S}}(f):=(L_f,R_f)$. The image of $\mathcal{C}_{\mathbb{C},\mathsf{S}}$ into $\mathit{Comult}_{\mathbb{C}}(\mathsf{S})$ is the set $\mathit{InnComult}_{\mathbb{C}}(\mathsf{S})$ of all {\em inner comultipliers} of $\mathsf{S}$. 

Recall that $C(S,I)=C^{op}(I,S)$ becomes a semigroup  under convolution which here means $g\bullet_{\delta} f:=S\xrightarrow{\delta}S\otimes S\xrightarrow{g\otimes f}I\otimes I\xrightarrow{\lambda_I=\rho_I}I$.

\begin{proposition}
Let $\mathsf{S}=(S,\delta)$ be a cosemigroup in $\mathbb{C}$.   $\mathcal{C}_{\mathbb{C},\mathsf{S}}\colon (C(S,I),\bullet_{\delta})\to (\mathit{Comult}_{\mathbb{C}}(\mathsf{S}),\star)$ is a homomorphism of semigroups and  $\mathit{InnComult}_{\mathbb{C}}(\mathsf{S})$ is a sub-semigroup of $(\mathit{Comult}_{\mathbb{C}}(\mathsf{S}),\star)$. 
\end{proposition}

Let $\mathsf{S}=(S,\delta)$ be a cosemigroup object in $\mathbb{C}$. Let $f,g\in C(S,I)$. Then, 
\begin{equation}
g\circ L_f=f\bullet_\delta g\ \mbox{and}\ g\circ R_f=g\bullet_\delta f.
\end{equation}
As a consequence one has the 
\begin{lemma}\label{lem:injectivity_of_can_map_for_left_and_right_non_degenerate_sem_for_cosem}
Let $\mathsf{S}=(S,\delta)$ be a cosemigroup object in $\mathbb{C}$. Assume that $(C(S,I),\bullet_{\delta})$ is non-degenerate (see Definition~\ref{def:non-degeneracy_of_a_semigroup}). Then, $\mathcal{C}_{\mathbb{C},\mathsf{S}}\colon  (C(S,I),\bullet_{\delta})\to (\mathit{Comult}_{\mathbb{C}}(\mathsf{S}),\star)$  is one-to-one. 
\end{lemma}

\section{Concretization by convolution}\label{sec:functoriality}

\subsection{The category of all internal semigroups}

Let $\mathbb{C}=(C,\otimes_C,I_C)$ and $\mathbb{D}=(D,\otimes_D,I_C)$ be monoidal categories, and $\mathbb{F}=(F,\Phi,\phi)\colon \mathbb{C}\to\mathbb{D}$ be a lax monoidal functor. 

\begin{proposition}\label{prop:mon_fun_lift_to_cats_of_actions}
For all semigroups $\mathsf{S}=(S,\mu),\mathsf{S}'=(S',\mu')$ in $\mathbb{C}$, $\mathbb{F}$ induces functors $\mathbf{Act}_{\mathsf{S}}(\mathbb{C})\to \mathbf{Act}_{\widetilde{\mathbb{F}}(\mathsf{S})}(\mathbb{D})$, ${}_{\mathsf{S}}\mathbf{Act}(\mathbb{C})\to {}_{\widetilde{\mathbb{F}}(\mathsf{S})}\mathbf{Act}(\mathbb{D})$ and ${}_{\mathsf{S}}\mathbf{Act}_{\mathsf{S}'}(\mathbb{C})\to {}_{\widetilde{\mathbb{F}}(\mathsf{S})}\mathbf{Act}_{\widetilde{\mathbb{F}}(\mathsf{S}')}(\mathbb{D})$, all denoted by $\tilde{F}$, and the following diagrams commute.
\begin{equation}
\xymatrix@R=1em{
\mathbf{Act}_{\mathsf{S}}(\mathbb{C})\ar[r]^{\tilde{F}} \ar[d]_{|-|_{\mathsf{S}}}& \mathbf{Act}_{\widetilde{\mathbb{F}}(\mathsf{S})}(\mathbb{D}) \ar[d]^{|-|_{\widetilde{\mathbb{F}}(\mathsf{S})}}&\ar[d]_{{}_{\mathsf{S}}|-|}{}_{\mathsf{S}}\mathbf{Act}(\mathbb{C})\ar[r]^{\tilde{F}} &\ar[d]^{{}_{\widetilde{\mathbb{F}}(\mathsf{S})}|-|} {}_{\widetilde{\mathbb{F}}(\mathsf{S})}\mathbf{Act}(\mathbb{D})\\
C\ar[r]_F&D&C\ar[r]_F&D\\
\ar[u]^{{}_{\mathsf{S}}|-|_{\mathsf{S}'}}{}_{\mathsf{S}}\mathbf{Act}_{\mathsf{S}'}(\mathbb{C}) \ar[r]_{\tilde{F}}& {}_{\widetilde{\mathbb{F}}(\mathsf{S})}\mathbf{Act}_{\widetilde{\mathbb{F}}(\mathsf{S}')}(\mathbb{D})\ar[u]_{{}_{\widetilde{\mathbb{F}}(\mathsf{S})}|-|_{\widetilde{\mathbb{F}}(\mathsf{S}')}}&&
}
\end{equation}
\end{proposition}

\begin{proof}
 Let $(c,\gamma)$ be a left $\mathsf{S}$-act. One claims that $\tilde{F}(c,\gamma):=(F(c),F(\gamma)\circ \Phi_{S,c})$ is a left $\widetilde{\mathbb{F}}(\mathsf{S})$-act. This follows from the commutativity of the following diagram. (The top left cell commutes by coherence and the bottom right cell commutes since $\gamma$ is a left action.)
\begin{equation}
\xymatrix{
F(S)\otimes_D (F(S)\otimes_D F(c))\ar[r]^{id_{F(S)}\otimes\Phi_{S,c}}\ar[d]_{\alpha^{-1}_{F(S),F(S),F(c)}} & \ar[d]^{\Phi_{S,S\otimes_C c}}\ar[rd]^{id_{F(S)}\otimes_D F(\gamma)}F(S)\otimes_D F(S\otimes_C c)&\\
\ar[d]_{\Phi_{S,S}\otimes_D id_{F(c)}}(F(S)\otimes_D F(S))\otimes_D F(c) &F(S\otimes_C (S\otimes_C c))\ar[d]_{F(\alpha^{-1}_{S,S,c})}\ar[rd]^{F(id_S\otimes_C\gamma)}&F(S)\otimes_D F(c)\ar[d]^{\Phi_{S,c}}\\
\ar[d]_{F(\mu)\otimes id_{F(c)}}F(S\otimes_C S)\otimes_D F(c) \ar[r]^{\Phi_{S\otimes_C S,c}}&\ar[d]_{F(\mu\otimes_C id_c)}F((S\otimes_C S)\otimes_C c)&F(S\otimes_C c)\ar[d]^{F(\gamma)}\\
\ar[r]_{\Phi_{S,c}}F(S)\otimes_D F(c) &\ar[r]_{F(\gamma)} F(S\otimes_C c) &F(c)
}
\end{equation}
Now let $f\in {}_{\mathsf{S}}\mathbf{Act}(\mathbb{C})((c,\gamma),(c',\gamma'))$. Then, $\tilde{F}(f):=F(f)$ belongs to ${}_{\widetilde{\mathbb{F}}(\mathsf{S})}\mathbf{Act}(\mathbb{D})((F(c),F(\gamma)\circ \Phi_{S,c}),(F(c'),F(\gamma')\circ\Phi_{S,c'}))$ as it is shown by the commutativity of the following diagram.
\begin{equation}
\xymatrix@R=1em{
F(c)\ar[r]^{F(f)} & F(c')\\
F(S\otimes_C c) \ar[u]^{F(\gamma)}\ar[r]^{F(id_S\otimes_c f)}& F(S\otimes_C c')\ar[u]_{F(\gamma')}\\
\ar[u]^{\Phi_{S,c}}F(S)\otimes_D F(c) \ar[r]_{id_{F(S)}\otimes_D F(f)}& F(S)\otimes_D F(c')\ar[u]_{\Phi_{S,c'}}
}
\end{equation}
One thus obtains the desired functor $\tilde{F}\colon {}_{\mathsf{S}}\mathbf{Act}(\mathbb{C})\to {}_{\widetilde{\mathbb{F}}(\mathsf{S})}\mathbf{Act}(\mathbb{D})$.

Let $(c,\rho)$ be a right $\mathsf{S}$-act. By symmetry, $\tilde{F}(c,\rho):=(F(c),F(\rho)\circ \Phi_{c,S})$ is a right $\widetilde{\mathbb{F}}(\mathsf{S})$-act and given $f\in \mathbf{Act}_{\mathsf{S}}(\mathbb{C})((c,\rho),(c',\rho'))$, $\tilde{F}(f):=F(f)\in \mathbf{Act}_{\widetilde{\mathbb{F}}(\mathsf{S})}(\mathbb{D})((F(c),F(\rho)\circ \Phi_{c,S}),(F(c'),F(\rho')\circ\Phi_{c',S}))$. One thus obtains the desired functor $\tilde{F}\colon \mathbf{Act}_{\mathsf{S}}(\mathbb{C})\to\mathbf{ Act}_{\widetilde{\mathbb{F}}(\mathsf{S})}(\mathbb{D})$.

Let $(c,\gamma,\rho)$ be a $\mathsf{S}\mhyphen\mathsf{S}'$-biact. That $(F(c),F(\gamma)\circ\Phi_{S,c},F(\rho)\circ \Phi_{S',c})$ is a $\widetilde{\mathbb{F}}(\mathsf{S})\mhyphen\widetilde{\mathbb{F}}(\mathsf{S}')$-biact follows from the commutativity of the following diagram. (The top right cell commutes by coherence.  The bottom left diagram commutes because $(c,\gamma,\rho)$ is a $\mathsf{S}\mhyphen\mathsf{S}'$-biact.)
\begin{equation}
\xymatrix@R=1em{
\ar[d]_{\alpha_{F(S),F(c),F(S')}}(F(S)\otimes_D F(c))\otimes_D F(S')\ar[r]^{\Phi_{S,c}\otimes_D id_{F(S')}} & \ar[d]_{\Phi_{S\otimes_C c, S'}}F(S\otimes_C c)\otimes_D F(S')\ar[rd]^{F(\gamma)\otimes_D id_{F(S')}}&\\
F(S)\otimes_D (F(c)\otimes_D F(S'))\ar[d]_{id_{F(S)}\otimes_D \Phi_{c,S'}} &\ar[d]_{F(\alpha_{S,c,S'})}F((S\otimes_C c)\otimes_C S')\ar[rd]_{F(\gamma\otimes_C id_{S'})}&\ar[d]^{\Phi_{c,S'}}F(c)\otimes_D F(S')\\
\ar[r]_{\Phi_{S,c\otimes_C S'}}\ar[d]_{id_{F(S)}\otimes_D F(\rho)}F(S)\otimes_D F(c\otimes_C S') &\ar[d]_{F(id_S\otimes\rho)}F(S\otimes_C(c\otimes_C S')) & F(c\otimes_C S')\ar[d]^{F(\rho)}\\
F(S)\otimes_D F(c) \ar[r]_{\Phi_{S,c}}& F(S\otimes_C c) \ar[r]_{F(\gamma)}& F(c)
}
\end{equation}
One  may thus define a functor $\tilde{F}\colon {}_{\mathsf{S}}Act_{\mathsf{S}'}(\mathbb{C})\to {}_{\widetilde{\mathbb{F}}(\mathsf{S})}Act_{\widetilde{\mathbb{F}}(\mathsf{S}')}(\mathbb{D})$ acting as $\tilde{F}((c,\gamma,\rho)\xrightarrow{f}(c'\gamma',\rho'))=(F(c),F(\gamma)\circ \Phi_{S,c},F(\rho)\circ \Phi_{c,S'})\xrightarrow{F(f)}(F(c'),F(\gamma')\circ \Phi_{S,c'},F(\rho')\circ \Phi_{c',S'})$. 
\end{proof}

\begin{lemma}\label{lem:un_petit_lem_teknik}
Let  $\mathsf{S}=(S,\mu),\mathsf{S}'=(S',\mu')$ be semigroups in $\mathbb{C}$. Let $(c,\gamma)$ be a left $\mathsf{S}$-act and let $(c',\rho')$ be a right $\mathsf{S}'$-act. Then, $\Phi_{c,c'}\in {}_{\widetilde{\mathbb{F}}(\mathsf{S})}Act_{\widetilde{\mathbb{F}}(\mathsf{S}')}(\mathbb{D})(\tilde{F}(c,\gamma)\otimes_D \tilde{F}(c',\rho'),\tilde{F}((c,\gamma)\otimes_C (c',\rho')))$. 
\end{lemma}

\begin{proof}
According to Proposition~\ref{prop:tens_prod_of_right_and_left_act}, $(c,\gamma)\otimes_C(c',\rho')=(c\otimes_C c',(\gamma\otimes id_{c'})\circ \alpha^{-1}_{S,c,c'},(id_c\otimes\rho')\circ\alpha_{c,c',S'})$ whence $\tilde{F}((c,\gamma)\otimes_C(c',\rho'))=(F(c\otimes_C c'), F(\gamma\otimes id_{c'})\circ F(\alpha^{-1}_{S,c,c'})\circ \Phi_{S,c\otimes_C c'},F(id_{c}\otimes\rho')\circ F(\alpha_{c,c',S'})\circ\Phi_{c\otimes_C c',S'})$. 

Now $\tilde{F}(c,\gamma)\otimes_D \tilde{F}(c',\rho')=(F(c),F(\gamma)\circ \Phi_{S,c})\otimes_D (F(c'),F(\rho')\circ \Phi_{c',S'})=
(F(c)\otimes_D F(c'),(F(\gamma)\otimes_D id_{F(c')})\circ (\Phi_{S,c}\otimes id_{F(c')})\circ \alpha^{-1}_{F(S),F(c),F(c')},(id_{F(c)}\otimes_D F(\rho'))\circ (id_{F(c)}\otimes_D \Phi_{c',S'})\circ \alpha_{F(c),F(c'),F(S')})$. 

The following diagram commutes. (The top cell commutes by naturality while the bottom cell commutes by coherence.)
\begin{equation}
\xymatrix@R=1em{
F(c) \otimes_D F(c')\ar[r]^{\Phi_{c,c'}} & F(c\otimes_C c')\\
\ar[u]^{F(\gamma)\otimes_D id_{F(c')}}F(S\otimes_C c)\otimes_D F(c') \ar[r]_{\Phi_{S\otimes_C c,c'}}& F((S\otimes_C c)\otimes_C c')\ar[u]_{F(\gamma\otimes_c id_{c'})}\\
\ar[u]^{\Phi_{S,c}\otimes_D id_{F(c')}}(F(S)\otimes_D F(c))\otimes_D F(c') & F(S\otimes_C (c\otimes_C c'))\ar[u]_{F(\alpha^{-1}_{S,c,c'})}\\
\ar[u]^{\alpha^{-1}_{F(S),F(c),F(c')}}F(S)\otimes_D (F(c)\otimes_D F(c')) \ar[r]_{id_{F(S)}\otimes_D \Phi_{c,c'}}&F(S)\otimes_D F(c\otimes_C c')\ar[u]_{\Phi_{S,c\otimes_C c'}}
}
\end{equation}
Whence $\Phi_{c,c'}$ is a left $\widetilde{\mathbb{F}}(\mathsf{S})$-morphism. Symmetrically, $\Phi_{c,c'}$ is a right $\widetilde{\mathbb{F}}(\mathsf{S}')$-morphism.  Consequently, $\Phi_{c,c'}\in {}_{\widetilde{\mathbb{F}}(\mathsf{S})}\mathbf{Act}_{\widetilde{\mathbb{F}}(\mathsf{S}')}(\mathbb{D})(\tilde{F}(c,\gamma)\otimes_D \tilde{F}(c',\rho'),\tilde{F}((c,\gamma)\otimes_C (c',\rho')))$.
\end{proof}

Let $\mathsf{S}=(S,\mu)$ be a semigroup in $\mathbb{C}$. Let $L\in \mathbf{Act}_{\mathsf{S}}(\mathbb{C})((S,\mu),(S,\mu))$. By Prop.~\ref{prop:mon_fun_lift_to_cats_of_actions}, $\tilde{F}(L)=F(L)\in \mathbf{Act}_{\widetilde{\mathbb{F}}(\mathsf{S})}(\mathbb{D})(\tilde{F}(S,\mu),\tilde{F}(S,\mu)))$. Now 
\begin{equation}
\begin{array}{lll}
\tilde{F}(\mathsf{g}_{\mathbb{C},\mathsf{S}}(L))\circ\Phi_{S,S}&=&F(\mu)\circ F(id_S\otimes_C L)\circ \Phi_{S,S}\\
&=&F(\mu)\circ\Phi_{S,S}\circ (id_{F(S)}\otimes_D F(L))\\
&=&\mathsf{g}_{\mathbb{D},\widetilde{\mathbb{F}}(\mathsf{S})}(\tilde{F}(L)).
\end{array}
\end{equation}
In a similar way one sees that for each $R\in {}_{\mathsf{S}}\mathbf{Act}(\mathbb{C})((S,\mu),(S,\mu))$, $\tilde{F}(\mathsf{d}_{\mathbb{C},\mathsf{S}}(R))=\mathsf{d}_{\mathbb{D},\tilde{\mathbb{F}}(\mathsf{S})}(\tilde{F}(R))$ as members of ${}_{\hat{\mathbb{F}}(\mathsf{S})}\mathbf{Act}_{\hat{\mathbb{F}}(\mathsf{S})}(\tilde{F}(S,\mu)\otimes_D \tilde{F}(S,\mu),\tilde{F}(S,\mu))$.

Now let $(L,R)\in \mathit{Mult}_{\mathbb{C}}(\mathsf{S})$. Then $(\tilde{F}(L),\tilde{F}(R))=(F(L),F(R))\in \mathit{Mult}_{\mathbb{D}}(\tilde{\mathbb{F}}(\mathsf{S}))$ since \begin{equation}
\begin{array}{lll}
\mathsf{g}_{\mathbb{D},\tilde{\mathbb{F}}(\mathsf{S})}(\tilde{F}(L))&=&\tilde{F}(\mathsf{g}_{\mathbb{C},\mathsf{S}}(L))\circ\Phi_{S,S}\\
&=&\tilde{F}(\mathsf{d}_{\mathbb{C},\mathsf{S}}(R))\circ\Phi_{S,S}\\
&=&\mathsf{d}_{\mathbb{D},\tilde{\mathbb{F}}(\mathsf{S})}(\tilde{F}(R)).
\end{array}
\end{equation}

Now one thus easily obtains the following
\begin{proposition}\label{prop:first_func_of_mult}
$(L,R)\mapsto (F(L),F(R))$ is a homomorphism $\mathit{Mult}_{disc}(\mathsf{S};\mathbb{F})$  of monoids from $\mathsf{Mult}_{\mathbb{C}}(\mathsf{S})$ to $\mathsf{Mult}_{\mathbb{D}}(\tilde{\mathbb{F}}(\mathsf{S}))$. If $F$ is faithful, then this homomorphism is one-to-one.
\end{proposition}

\begin{corollary}
Let $(\mathbb{C},\sigma)$ be a symmetric monoidal category. Then for each semigroup $\mathsf{S}$ in $\mathbb{C}$, $\mathsf{Mult}_{\mathbb{C}}(\mathsf{S})^{op}\simeq \mathsf{Mult}_{\mathbb{C}}(\mathsf{S}^{op})$, $(L,R)\mapsto (R,L)$. In particular if $\mathsf{S}$ is commutative, then $(L,R)\mapsto (R,L)$ provides an involutive anti-automorphism of $\mathsf{Mult}_{\mathbb{C}}(\mathsf{S})$. 
\end{corollary}
\begin{proof}
It suffices to apply Prop.~\ref{prop:first_func_of_mult} with the monoidal functor $\mathbb{\Sigma}:=(id_C,\sigma,id_I)\colon \mathbb{C}\to\mathbb{C}^t$ to obtain the homomorphism $\mathit{Mult}_{disc}(\mathsf{S};\mathbb{\Sigma})\colon \mathsf{Mult}_{\mathbb{C}}(\mathsf{S})\to \mathsf{Mult}_{\mathbb{C}^t}(\mathsf{S}^{op})$, and then to note that $(\mathsf{S}^{op})^{op}=\mathsf{S}$ and $(\mathbb{C}^{t})^t=\mathbb{C}$ so that $\mathbb{\Sigma}\colon \mathbb{C}^t\to\mathbb{C}$, which then provides a homomorphism of monoids from $\mathsf{Mult}_{\mathbb{C}^t}(\mathsf{S}^{op})$ to $\mathsf{Mult}_{\mathbb{C}}(\mathsf{S})$, inverse from the first one. One concludes using Lemma~\ref{lem:formal_symmetry}. 
\end{proof}

\begin{example}
Assume that $\mathbb{C}'$ is a monoidal subcategory of $\mathbb{C}$, that is, the canonical inclusion functor is strictly monoidal. Then for any semigroup $\mathsf{S}$ in $\mathbb{C}'$, $\mathsf{Mult}_{\mathbb{C}'}(\mathsf{S})\subseteq \mathsf{Mult}_{\mathbb{C}}(\mathsf{S})$. If the inclusion is full, then one has even $\mathsf{Mult}_{\mathbb{C}'}(\mathsf{S})= \mathsf{Mult}_{\mathbb{C}}(\mathsf{S})$.
\end{example}

\begin{lemma}\label{lem:functoriality_mult_with_respect_to_monoidal_functors}
Let $\mathsf{S}$ be a semigroup in $\mathbb{C}$. 
\begin{enumerate}
\item $\mathit{Mult}_{disc}(\mathsf{S};{\mathbb{Id}_{\mathbb{C}}})=id_{\mathsf{Mult}_{\mathbb{C}}(\mathsf{S})}$. 
\item Let $\mathbb{F}\colon \mathbb{A}\to\mathbb{B}$ and $\mathbb{G}\colon \mathbb{B}\to\mathbb{C}$ be lax monoidal functors. Then one has $\mathit{Mult}_{disc}(\mathsf{S};\mathbb{G}\circ \mathbb{F})=\mathit{Mult}_{disc}(\widetilde{\mathbb{F}}(\mathsf{S});{\mathbb{G}})\circ \mathit{Mult}_{disc}(\mathsf{S};{\mathbb{F}})\colon \mathsf{Mult}_{\mathbb{A}}(\mathsf{S})\to \mathsf{Mult}_{\mathbb{C}}(\widetilde{\mathbb{G}\circ\mathbb{F}}(\mathsf{S}))$.
\end{enumerate}
\end{lemma}

What kind of functoriality is stated above?  By Grothendieck's construction (see e.g.~\cite{Borceux2}) the functor $\mathit{Sem}\colon \mathbf{MonCat}\to \mathbf{Cat}$ from Section~\ref{sec:prerequisites}  provides a split opfibration over $\mathbf{MonCat}$, that is, an opfibred category $\mathbf{IntSem}\xrightarrow{\Pi}\mathbf{MonCat}$ with a choice of cartesian morphisms, where $\mathbf{IntSem}$ is the category of {\em all} internal semigroups. In details, $\mathbf{IntSem}$ has objects the pairs $(\mathbb{C},\mathsf{S})$ consisting of a monoidal category $\mathbb{C}$ together with a semigroup $\mathsf{S}$ in $\mathbb{C}$, morphisms $(\mathbb{C},\mathsf{S})\xrightarrow{(\mathbb{F},f)}(\mathbb{D},\mathsf{T})$ the pairs $(\mathbb{F},f)$ consisting of a monoidal functor $\mathbb{F}\colon\mathbb{C}\to\mathbb{D}$ and a morphism $f\in \mathbf{Sem}(\mathbb{D})(\widetilde{\mathbb{F}}(\mathsf{S}),\mathsf{T})$, composition: $(\mathbb{C},\mathsf{S})\xrightarrow{(\mathbb{F},f)}(\mathbb{D},\mathsf{T})\xrightarrow{(\mathbb{G},g)}(\mathbb{E},\mathsf{U})=(\mathbb{C},\mathsf{S})\xrightarrow{(\mathbb{G}\circ\mathbb{F},g\circ \widetilde{\mathbb{G}}(f))}(\mathbb{E},\mathsf{U})$, and identities $id_{(\mathbb{C},\mathsf{S})}:=(\mathbb{C},\mathsf{S})\xrightarrow{(\mathbb{id}_{\mathbb{C}},id_{\mathsf{S}})}(\mathbb{C},\mathsf{S})$. Moreover the projection functor $\Pi\colon \mathbf{IntSem}\to \mathbf{MonCat}$ is a split opfibration given by $\Pi((\mathbb{C},\mathsf{S})\xrightarrow{(\mathbb{F},f)}(\mathbb{D},\mathsf{T})):=\mathbb{C}\xrightarrow{\mathbb{F}}\mathbb{D}$ and with cleavage described as follows: given a monoidal functor $\mathbb{F}\colon \mathbb{C}\to\mathbb{D}$ and a semigroup $\mathsf{S}$ in $\mathbb{C}$, $(\mathbb{C},\mathsf{S})\xrightarrow{(\mathbb{F},id_{\widetilde{\mathbb{F}}(\mathsf{S})})}(\mathbb{D},\widetilde{\mathbb{F}}(\mathsf{S}))$ is an opcartesian morphism over $\mathbb{F}$ with domain $(\mathbb{C},\mathsf{S})$. The fibre at a monoidal category $\mathbb{C}$, that is, the subcategory of $\mathbf{IntSem}$ consisting of morphisms of the form $(\mathbb{C},\mathsf{S})\xrightarrow{(id_{\mathbb{C}},f)}(\mathbb{C},\mathsf{T})$, is isomorphic to $\mathbf{Sem}(\mathbb{C})$. 

$\mathbf{IntSem}$ has a subcategory of particular interest for us: let $\mathbf{IntSem}^{disc}$ be the subcategory of $\mathbf{IntSem}$ consisting only of its opcartesian morphisms. $\mathbf{IntSem}^{disc}$ is in fact (isomorphic to) the category of elements (see~\cite{MacLane}) of the functor $\mathbf{MonCat}\xrightarrow{\mathit{Sem}}\mathbf{Cat}\xrightarrow{\mathit{Ob}}\mathbf{Set}$ and thus corresponds to a discrete opfibration. Lemma~\ref{lem:functoriality_mult_with_respect_to_monoidal_functors} tells us that there is a functor $\mathit{Mult}_{disc}\colon \mathbf{IntSem}^{disc}\to \mathbf{Mon}$, acting as $\mathit{Mult}_{disc}((\mathbb{C},\mathsf{S})\xrightarrow{(\mathbb{F},id_{\widetilde{\mathbb{F}}(\mathsf{S})})}(\mathbb{D},\widetilde{\mathbb{F}}(\mathsf{S})):=\mathsf{Mult}_{\mathbb{C}}(\mathsf{S})\xrightarrow{\mathit{Mult}_{disc}(\mathsf{S};{\mathbb{F}})}\mathsf{Mult}_{\mathbb{D}}(\widetilde{\mathbb{F}}(\mathsf{S}))$.

\subsection{The convolution functor as an initial object}

\subsubsection{The domain fibration}
Let $\mathbb{F}=(F,\Phi,\phi),\mathbb{G}=(G,\Psi,\psi)$ be monoidal functors from $\mathbb{C}=(C,\otimes_C,I)$ to $\mathbb{D}=(D,\otimes_D,J)$. Let $\mathit{MonNat}(\mathbb{F},\mathbb{G})$ be the set of all monoidal natural transformations from $\mathbb{F}$ to $\mathbb{G}$, that is, $\alpha\in\mathit{MonNat}(\mathbb{F},\mathbb{G})$ (also written $\alpha\colon \mathbb{F}\Rightarrow\mathbb{G}\colon \mathbb{C}\to\mathbb{D}$) iff $\alpha\in \mathit{Nat}(F,G)$ and the following diagram commutes.
\begin{equation}\label{diag:mon_nat}
\xymatrix@R=1em{
Fc \otimes_D Fd\ar[r]^{\alpha_c\otimes_D\alpha_d}\ar[d]_{\Phi_{c,d}}&Gc\otimes_D Gd\ar[d]^{\Psi_{c,d}}&&&\ar[ld]_{\phi}\ar[rd]^{\psi}J&\\
F(c\otimes_C d)\ar[r]_{\alpha_{c\otimes_C d}}&G(c\otimes_C d)&&FI\ar[rr]_{\alpha_I}&&GI
}
\end{equation}
Note that  a monoidal natural transformation $\alpha\colon \mathbb{F}\Rightarrow\mathbb{G}\colon \mathbb{C}\to\mathbb{D}$ lifts to a natural transformation $\tilde{\alpha}\colon \widetilde{\mathbb{F}}\Rightarrow\widetilde{\mathbb{G}}\colon \mathbf{Sem}(\mathbb{C})\to\mathbf{Sem}(\mathbb{D})$ with $\tilde{\alpha}_{\mathsf{S}}=\alpha_S$, where $S$ stands for the underlying object of a semigroup object $\mathsf{S}$.

Let $\mathbf{MonCat}\Downarrow \mathbb{Set}$ be the category with 
\begin{enumerate}
\item objects the lax monoidal functors $\mathbb{C}\xrightarrow{\mathbb{F}}\mathbb{Set}$ for varying monoidal categories $\mathbb{C}$,
\item hom-sets $(\mathbf{MonCat}\Downarrow \mathbb{Set})((\mathbb{C}\xrightarrow{\mathbb{F}}\mathbb{Set}),(\mathbb{C}'\xrightarrow{\mathbb{F}'}\mathbb{Set}))$ the pairs $(\mathbb{H},\alpha)$, where $\mathbb{C}\xrightarrow{\mathbb{H}}\mathbb{D}$ is a monoidal functor, and $\alpha\colon \mathbb{F}\Rightarrow\mathbb{F}'\circ\mathbb{H}\colon \mathbb{C}\to\mathbb{Set}$ is a monoidal natural transformation.\end{enumerate}
There is an obvious projection functor $\mathbf{MonCat}\Downarrow\mathbb{Set}\xrightarrow{\mathit{dom}}\mathbf{MonCat}$ which is a split fibration as it arises from the Grothendieck's construction applied to  the functor $(-)^*\colon \mathbf{MonCat}^{op}\to\mathbf{Cat}$ given by $\mathbb{C}^*:=\mathbb{Set}^{\mathbb{C}}$, that is, $\mathbb{C}^*$ is the category whose objects are monoidal functors $\mathbb{C}\to\mathbb{Set}$ and whose morphisms are monoidal natural transformations with vertical composition of the underlying natural transformations, and for a monoidal functor $\mathbb{H}\colon \mathbb{C}\to\mathbb{D}$, $\mathbb{H}^*\colon \mathbb{Set}^{\mathbb{D}}\to\mathbb{Set}^{\mathbb{C}}$ is the functor given by $(\mathbb{F}\colon \mathbb{D}\to\mathbb{Set})\mapsto (\mathbb{F}\circ\mathbb{H}\colon \mathbb{C}\to\mathbb{Set})$ and $(\alpha\colon \mathbb{F}\Rightarrow\mathbb{G}\colon \mathbb{D}\to\mathbb{Set})\mapsto (\alpha_{\mathbb{H}}\colon \mathbb{F}\circ\mathbb{H}\Rightarrow\mathbb{G}\circ\mathbb{H}\colon \mathbb{C}\to\mathbb{Set})$, where the underlying natural transformation of $\alpha_{\mathbb{H}}$ is $\alpha_H$, with $H$  the underlying functor of $\mathbb{H}$.

\subsubsection{The Yoneda lemma for monoidal functors}

Let $\mathbb{C}$ be a monoidal category.

Let us  define the following monoidal functor $\mathbb{G}_I=(G_I,\Gamma,\gamma)\colon \mathbb{C}\to \mathbb{C}^{op}\times\mathbb{C}$ as follows: its underlying functor is $G_I\colon C\to C^{op}\times C$, $G_I(c\xrightarrow{f}d):=(I\xrightarrow{id_I}I,c\xrightarrow{f}d)$, $\Gamma_{c,d}\colon G_I(c)\otimes G_I(d)=(I\otimes I,c\otimes d)\xrightarrow{(\lambda_I^{-1},id_{c\otimes d})}G_I(c\otimes d)=(I,c\otimes d)$, and $\gamma=(I,I)\xrightarrow{(id_I,id_I)}G_I(I)=(I,I)$. The composite monoidal functor $\mathbb{Conv}\circ\mathbb{G}_I$ is denoted $\mathbb{Conv}_I=(C(I,-),\Theta^C,\theta^C)$ and of course, $\widetilde{\mathbb{Conv}_I}=\mathsf{Conv}_I$, since $\Theta^C_{c,d}\colon C(I,c)\times C(I,d)\to C(I,c\otimes d), (f,g)\mapsto (f\otimes g)\circ\lambda_I^{-1}$ and $\theta^C\colon 1\to C(I,I)$, $0\mapsto id_I$. 

The Yoneda Lemma \cite{MacLane} tells us that for each functor $F\colon C\to\mathbf{Set}$ and each $C$-object $c$, $\mathit{Nat}(C(c,-),F)\simeq Fc$ under $\alpha\mapsto \alpha_c(id_c)$, with inverse $x\in Fc\mapsto \mathbf{x}$ where $\mathbf{x}_d(f):=Ff(x)$, $f\in C(c,d)$. 

\begin{lemma}\label{lem:Yoneda_for_mon_nat}
Let $\mathbb{F}=(F,\Phi,\phi)\colon \mathbb{C}\to\mathbb{Set}$ be a monoidal functor. Then,  $\mathit{MonNat}(\mathbb{Conv}_I,\mathbb{F})\simeq \{\, \phi(0)\,\}$. Put another way, $\mathbb{Conv}_I$ is an initial object in the category $\mathbb{Set}^{\mathbb{C}}$.
\end{lemma}

\begin{proof}
Let $\alpha\colon \mathbb{Conv}_I\Rightarrow \mathbb{F}$ be a monoidal natural transformation. Compatibility with the unit constraints (the rightmost diagram in Diag.~(\ref{diag:mon_nat})) implies that $\alpha_I(id_I)=\phi(0)$. Conversely one has to check that $\mathbf{x}\colon \mathbb{Conv}_I\Rightarrow\mathbb{F}$ for $x:=\phi(0)$. As just explained above, compatibility with the unit constraint is immediate so it remains to check that the leftmost diagram in Diag.~(\ref{diag:mon_nat}) commutes. By direct inspection one first observes that the following diagram commutes. 
\begin{equation}\label{diag:mon_cat_for_x}
\xymatrix@R=1em{
C(I,I)\times C(I,I)\ar[r]^-{\mathbf{x}_I\times\mathbf{x}_I}\ar[d]_{\Theta^C_{I,I}}&FI\times FI\ar[d]^{\Phi_{I,I}}\\
C(I,I\otimes I)\ar[r]_{\mathbf{x}_{I\otimes I}}&F(I\otimes I)
}
\end{equation}
Now let $(f,g)\in C(I,c)\times C(I,d)$. Then, 
\begin{equation}
\begin{array}{lll}
\Phi_{c,d}(\mathbf{x}_c(f),\mathbf{x}_d(g))&=&(\Phi_{c,d}\circ (F(f)\times F(g))\circ (\mathbf{x}_I\times\mathbf{x}_I))(id_I,id_I)\\
&=&(F(f\otimes g)\circ\Phi_{I,I}\circ (\mathbf{x}_I\times \mathbf{x}_I))(id_I,id_I)\\
&&\mbox{(by naturality of $\Phi$)}\\
&=&(F(f\otimes g)\circ \mathbf{x}_{I\otimes I}\circ\Theta^C_{I,I})(id_I,id_I)\\
&&\mbox{(by commutativity of Diag.~(\ref{diag:mon_cat_for_x}))}\\
&=&(\mathbf{x}_{c\otimes d}\circ C(I,f\otimes g)\circ \Theta^C_{I,I})(id_I,id_I)\\
&&\mbox{(by naturality of $\mathbf{x}$)}\\
&=&\mathbf{x}_{c\otimes d}(\Theta^C_{c,d}(f,g)).
\end{array}
\end{equation}
\end{proof}

\subsubsection{The convolution section functor}

Let $\mathit{Sect}(\mathit{dom})$ be the set of all sections of the functor $\mathit{dom}$, that is, $\mathcal{X}\in \mathit{Sect}(\mathit{dom})$ iff $\mathcal{X}\colon \mathbf{MonCat}\to\mathbf{MonCat}\Downarrow\mathbb{Set}$ such that $\mathit{dom}\circ\mathcal{X}=id_{\mathbf{MonCat}}$. So for a section $\mathcal{X}$ one has $\mathcal{X}(\mathbb{C}):=\mathbb{X}(\mathbb{C})=(X(\mathbb{C}),\Psi^{\mathcal{X},\mathbb{C}},\psi^{\mathcal{X},\mathbb{C}})\colon \mathbb{C}\to\mathbb{Set}$ and given a monoidal functor $\mathbb{F}\colon \mathbb{C}\to\mathbb{D}$, $\mathcal{X}(\mathbb{F}):=(\mathbb{F},\mathbb{x}^{\mathcal{X},\mathbb{F}})\colon \mathbb{X}(\mathbb{C})\to\mathbb{X}(\mathbb{D})$ in $\mathbf{MonCat}\Downarrow\mathbb{Set}$. 

$\mathit{Sect}(\mathit{dom})$ is the set of objects of a subcategory $\mathbf{Sect}(\mathit{dom})$ of $(\mathbf{MonCat}\Downarrow\mathbb{Set})^{\mathbf{MonCat}}$, whose hom-set  $\mathbf{Sect}(dom)(\mathcal{X},\mathcal{X}')$ consists of the natural transformations $\alpha\colon \mathcal{X}\Rightarrow \mathcal{X}'\colon \mathbf{MonCat}\to \mathbf{MonCat}\Downarrow\mathbb{Set}$ such that for each monoidal category $\mathbb{C}$, $\mathit{dom}(\alpha_{\mathbb{C}})=id_{\mathbb{C}}$. Whence $\alpha_{\mathbb{C}}=(id_{\mathbb{C}},\alpha^m(\mathbb{C}))\colon \mathbb{X}(\mathbb{C})\to \mathbb{X}'(\mathbb{C})$ where $\alpha^m(\mathbb{C})\colon \mathbb{X}(\mathbb{C})\Rightarrow \mathbb{X}'(\mathbb{C})\colon \mathbb{C}\to\mathbb{Set}$ is a monoidal natural transformation. Note that by naturality, for each monoidal functor $\mathbb{F}\colon\mathbb{C}\to\mathbb{D}$, and for each $C$-object $c$, \begin{equation}\label{diag:relation_between_x_and_alpha}
\mathbb{x}_c^{\mathcal{X}',\mathbb{F}}\circ \alpha^{m}(\mathbb{C})_c=\alpha^m(\mathbb{D})_{Fc}\circ\mathbb{x}^{\mathcal{X},\mathbb{F}}_c.
\end{equation} 

Among the sections of $\mathit{dom}$ there is a canonical one which is now described. Let $\mathbb{F}=(F,\Phi,\phi)\colon \mathbb{C}=(C,\otimes,I)\to\mathbb{D}=(D,\otimes,J)$ be a monoidal functor. 
According to Lemma~\ref{lem:Yoneda_for_mon_nat}, $\mathit{MonNat}(\mathbb{Conv}_I,\mathbb{Conv}_J\circ\mathbb{F})$ has only one element namely $\mathbb{F}^\sharp:=\mathbf{x}$ with $x=D(J,\phi)(\theta^D(0))=\phi$. One then defines  the {\em convolution section} $\mathpzc{Conv}\colon \mathbf{MonCat}\to \mathbf{MonCat}\Downarrow \mathbb{Set}$ by $\mathpzc{Conv}(\mathbb{C}):=\mathbb{Conv}_I\colon \mathbb{C}\to\mathbb{Set}$ and given $\mathbb{F}\colon \mathbb{C}\to\mathbb{D}$, $\mathpzc{Conv}(\mathbb{F}):=(\mathbb{F},\mathbb{F}^\natural)$ where $\mathbb{F}^\natural$ is as above.

Let $\alpha\in \mathbf{Sect}(\mathit{dom})(\mathpzc{Conv},\mathcal{X})$. Then, for each monoidal category $\mathbb{C}$, $\alpha^{m}(\mathbb{C})\in \mathit{MonNat}(\mathbb{Conv}_I,\mathcal{X})=\{\, \mathbf{x}^{\mathcal{X},\mathbb{C}}\,\}$ (again by Lemma~\ref{lem:Yoneda_for_mon_nat}) with $x^{\mathcal{X},\mathbb{C}}:=\psi^{\mathcal{X},\mathbb{c}}(0)$. It then follows easily that $\mathpzc{Conv}$ is an initial object in $\mathbf{Sect}(\mathit{dom})$.

\subsubsection{Relation with $\mathbf{IntSem}$}

Let $\mathcal{X}\in \mathit{Sect}(\mathit{dom})$. Let $\mathsf{S}=(S,\mu_S)$ (resp. $\mathsf{T}=(T,\mu_T)$) be a semigroup in $\mathbb{C}=(C,\otimes,I)$ (resp. $\mathbb{D}=(D,\otimes,J)$). Let $(\mathbb{F},f)\in \mathbf{IntSem}((\mathbb{C},\mathsf{S}),(\mathbb{D},\mathsf{T}))$. Then, one has a map 
$\mathsf{X}(\mathbb{F},f):=X(\mathbb{C})(S)\xrightarrow{\mathbb{x}_S^{\mathcal{X},\mathbb{F}}}X(\mathbb{D})(FS)\xrightarrow{X(\mathbb{D})(f)}X(\mathbb{D})(T)$. 

\begin{example}
In the case where $\mathcal{X}=\mathpzc{Conv}$, one has $\mathsf{Conv}(\mathbb{F},f)= C(I,S)\xrightarrow{\mathbb{F}^\sharp_S}D(J,FS)\xrightarrow{D(J,f)}D(J,T)$, $g\mapsto f\circ F(g)\circ\phi$.
\end{example}

Since $\widetilde{\mathbb{X}(\mathbb{D})}(f)=X(\mathbb{D})(f)\in\mathbf{Sem}(\widetilde{\mathbb{X}(\mathbb{D})}(\widetilde{\mathbb{F}}\mathsf{S}),\widetilde{\mathbb{X}(\mathbb{D})}\mathsf{T})$ and $\widetilde{\mathbb{x}^{\mathcal{X},\mathbb{F}}}_{\mathsf{S}}=\mathbb{x}^{\mathcal{X},\mathbb{F}}_S\in \mathbf{Sem}(\widetilde{\mathbb{X}(\mathbb{C})}\mathsf{S},\widetilde{\mathbb{X}(\mathbb{D})}(\widetilde{\mathbb{F}}\mathsf{S}))$, one has $\mathsf{X}(\mathbb{F},f)\in \mathbf{Sem}(\widetilde{\mathbb{X}(\mathbb{C})}\mathsf{S},\widetilde{\mathbb{X}(\mathbb{D})}\mathsf{T})$.

\begin{lemma}
With the above notation $\mathsf{X}$ is a functor from $\mathbf{IntSem}$ to $\mathbf{Sem}$ with $\mathsf{X}(\mathbb{C},\mathsf{S}):=\widetilde{\mathbb{X}(\mathbb{C})}\mathsf{S}$.
\end{lemma}
\begin{proof}
Let $\mathcal{X}$ be a section of $\mathit{dom}$. Using functoriality of $\mathcal{X}$ it is easy to see that $\mathsf{X}(id_{\mathbb{C}},id_{\mathsf{S}})=id_{\mathsf{X}(\mathbb{C},\mathsf{S})}$. Let $\mathbb{F}\colon \mathbb{A}\to\mathbb{B}$ and $\mathbb{G}\colon \mathbb{B}\to\mathbb{C}$ be monoidal functors. Let $f\in \mathbf{Sem}(\mathbb{B})(\widetilde{\mathbb{F}}\mathsf{S},\mathsf{T})$ and $g\in \mathbb{Sem}(\mathbb{C})(\widetilde{\mathbb{G}}\mathsf{T},\mathsf{U})$. One has $\mathsf{X}((\mathbb{G},g)\circ (\mathbb{F},f))=\mathsf{X}(\mathbb{G}\circ\mathbb{F},g\circ G(f))=X(\mathbb{C})(g\circ G(f))\circ \mathbb{x}_{{S}}^{\mathcal{X},\mathbb{G}\circ\mathbb{F}}=X(\mathbb{C})(g)\circ X(\mathbb{C})(G(f))\circ \mathbb{x}_{{S}}^{\mathcal{X},\mathbb{G}\circ\mathbb{F}}=X(\mathbb{C})(g)\circ X(\mathbb{C})(G(f))\circ \mathbb{x}^{\mathcal{X},\mathbb{G}}_{FS}\circ \mathbb{x}^{\mathcal{X},\mathbb{F}}_S$ (by functoriality) $=X(\mathbb{C})(g)\circ \mathbb{x}_{\mathsf{T}}^{\mathcal{X},\mathbb{G}}\circ X(\mathbb{B})(f)\circ \mathbb{x}^{\mathcal{X},\mathbb{F}}_S$ (by naturality of $\mathbb{x}^{\mathcal{X},\mathbb{G}}$) $=\mathsf{X}(\mathbb{G},g)\circ\mathsf{X}(\mathbb{F},f)$.
\end{proof}

\begin{remark}
When $\mathcal{X}=\mathpzc{Conv}$, one let $\mathsf{X}$ be  $\mathsf{Conv}\colon \mathbf{IntSem}\to\mathbf{Sem}$. In details, $\mathsf{Conv}$ acts as $((\mathbb{C},\mathsf{S})\xrightarrow{(\mathbb{F},f)}(\mathbb{D},\mathsf{T}))\mapsto (\mathsf{Conv}_I(\mathsf{S})\xrightarrow{D(\phi,f)\circ F_{I,S}}\mathsf{Conv}_J(\mathsf{T}))$, that is, $\mathsf{Conv}(\mathbb{F},f)(g)=f\circ F(g)\circ\phi\in D(J,T)$ for $g\in C(I,S)$. 
\end{remark}

Now let $\mathcal{X},\mathcal{X}'\in \mathit{Sect}(dom)$ and let $\alpha\in \mathbf{Sect}(dom)(\mathcal{X},\mathcal{X}')$. Then, of course, for each semigroup $\mathsf{S}=(S,\mu_S)$ of $\mathbb{C}$, $\widetilde{\alpha^m(\mathbb{C})}_{\mathsf{S}}=\alpha^m(\mathbb{C})_S\in\mathbf{Sem}(\widetilde{\mathbb{X}(\mathbb{C})}\mathsf{S},\widetilde{\mathbb{X}'(\mathbb{C})}\mathsf{S})=\mathbf{Sem}(\mathsf{X}(\mathbb{C},\mathsf{S}),\mathsf{X}'(\mathbb{C},\mathsf{S}))$. The commutativity of the two internal cells of the diagram below, where $(\mathbb{F},f)\colon (\mathbb{C},\mathsf{S})\to(\mathbb{D},\mathsf{T})$ is a $\mathbf{IntSem}$-morphism, shows that $\widetilde{\alpha^m}:=(\widetilde{\alpha^m(\mathbb{C})}_{\mathsf{S}})_{(\mathbb{C},\mathsf{S})}\colon \mathsf{X}\Rightarrow \mathsf{X}'\colon \mathbf{IntSem}\to\mathbf{Sem}$ is a natural transformation.  (The top square commutes by Eq.~(\ref{diag:relation_between_x_and_alpha}) while the bottom square commutes by naturality of $\alpha^m(\mathbb{D})$.)

\begin{equation}
\xymatrix{
{{X}(\mathbb{C})}\mathsf{S}\ar[d]_{\mathbb{x}_{S}^{\mathcal{X},\mathbb{F}}}\ar[r]^{\alpha^m(\mathbb{C})_S}&{{X}'(\mathbb{C})}\mathsf{S}\ar[d]^{\mathbb{x}_{S}^{\mathcal{X}',\mathbb{F}}}\\
\ar[d]_{X(\mathbb{D})(f)}{X}(\mathbb{D})FS \ar[r]_{\alpha^m(\mathbb{D})_{FS}}&{X}'(\mathbb{D})FS\ar[d]^{X'(\mathbb{D})(f)}\\
{X}(\mathbb{D})\mathsf{T}\ar[r]_{\alpha^m(\mathbb{D})_{T}}&{\mathbb{X}'(\mathbb{D})}\mathsf{T}
}
\end{equation}

The following result then is clear.
\begin{lemma}
The correspondences $\mathcal{X}\in \mathit{Sect}(\mathit{dom})\mapsto \mathsf{X}\in \mathit{Ob}(\mathbf{Sem}^{\mathbf{IntSem}})$ and $\alpha\in \mathbf{Sect}(\mathit{dom})(\mathcal{X},\mathcal{X}')\mapsto \widetilde{\alpha^m}\colon \mathsf{X}\Rightarrow \mathsf{X}'$ provide a functor from $\mathbf{Sect}(dom)$ to $\mathbf{Sem}^{\mathbf{IntSem}}$.
\end{lemma}

\subsection{The concretization homomorphism}\label{sec:concretization}

\begin{definition}\label{def:concretization_morphism}
Given a semigroup $\mathsf{S}=(S,\mu)$ in a monoidal category $\mathbb{C}$, according to Prop.~\ref{prop:first_func_of_mult} one has a homomorphism of monoids, called the {\em concretization homomorphism}, $$\mathit{Conc}_{\mathbb{C}}(\mathsf{S}):=\mathsf{Mult}_{\mathbb{C}}(\mathsf{S})\xrightarrow{\mathit{Mult}_{disc}(\mathsf{S};\mathbb{Conv}_I)}\mathsf{Mult}_{\mathbb{Set}}(\mathsf{Conv}_I(\mathsf{S}))=\mathsf{TrHull}(\mathsf{Conv}_{I}(\mathsf{S}))$$ which acts as $(L,R)\mapsto (C(id_I,L),C(id_I,R))$. Of course, $C(id_I,L)\colon C(I,S)\to C(I,S)$, $f\mapsto L\circ f$ and $C(id_I,R)\colon C(I,S)\to C(I,S)$, $f\mapsto R\circ f$. 
\end{definition}

\begin{remark}
When $C(I,-)$ is faithful, then $\mathsf{Mult}_{\mathbb{C}}(\mathsf{S})$ is (isomorphic to) a submonoid of $\mathsf{TrHull}(\mathsf{Conv}_I(\mathsf{S}))$. 
\end{remark}

\begin{definition}
A semigroup $\mathsf{S}$  in $\mathbb{C}$. $\mathsf{S}$ is said to be {\em concrete} when the concretization homomorphism $\mathit{Conc}_{\mathbb{C}}(\mathsf{S})\colon \mathsf{Mult}_{\mathbb{C}}(\mathsf{S})\to \mathsf{TrHull}(\mathsf{Conv}_I(\mathsf{S}))$ is onto. 
\end{definition}

Let us finally consider the composite $\mathsf{Conv}_I(\mathsf{S})\xrightarrow{\mathcal{M}_{\mathbb{C},\mathsf{S}}}|\mathsf{Mult}_{\mathbb{C}}(\mathsf{S})|\xrightarrow{\mathit{Conc}_{\mathbb{C}}(\mathsf{S})}|\mathsf{TrHull}(\mathsf{Conv}_I(\mathsf{S}))|$. It is the map $f\mapsto (C(I,L_f),C(I,R_f))$. Using Eqs.~(\ref{eq:computations_for_concretization_homomorphism}) and (\ref{eq:computations_for_concretization_homomorphism_2}), $C(I,L_f)=\mathfrak{L}_f$, and $C(I,R_f)=\mathfrak{R}_f$. Therefore, the above composition is nothing but $\mathfrak{M}_{ \mathsf{Conv}_{I}(\mathsf{S})}\colon \mathsf{Conv}_{I}(\mathsf{S})\to |\mathsf{TrHull}(\mathsf{Conv}_I(\mathsf{S}))|$.

\begin{example}\label{ex:concrete_sem}
\begin{enumerate}
\item Every ordinary semigroup  is concrete (as it follows easily from Example~\ref{ex:generalized_elements}).

\item The underlying semigroup object of any monoid object in any monoidal category is concrete. Indeed according to Prop.~\ref{prop:mult_alg_of_a_monoid}, for each monoid $\mathsf{M}$, $\mathcal{M}_{\mathbb{C},\mathsf{M}}\colon \mathsf{Conv}_I(|\mathsf{M}|)\simeq \mathsf{Mult}_{\mathbb{C}}(\mathsf{M})$. Moreover $\mathfrak{M}_{\mathsf{Conv}_I(\mathsf{M})}\colon \mathsf{Conv}_{I}(\mathsf{M})\simeq \mathsf{Mult}(|\mathsf{Conv}_I(\mathsf{M})|)$ is also an isomorphism of monoids (see the Introduction). Consequently, $\mathit{Conc}_{\mathbb{C}}(|\mathsf{M}|)=\mathfrak{M}_{\mathsf{Conv}_I(|\mathsf{M}|)}\circ \mathcal{M}^{-1}_{\mathbb{C},|\mathsf{M}|}=|\mathfrak{M}_{\mathsf{Conv}_I(\mathsf{M})}\circ \mathcal{M}^{-1}_{\mathbb{C},\mathsf{M}}|$ is an isomorphism of semigroups. 

\item Let $R$ be a unital and commutative ring. Then, any $R$-algebra $\mathsf{A}=(A,\mu)$, that is, a semigroup in ${}_R\mathbb{Mod}=({}_R\mathbf{Mod},\otimes_R,R)$, faithful as a module over itself, is concrete. Indeed, the usual forgetful functor ${}_{R}\mathbf{Mod}\to \mathbf{Set}$ is represented by the free module of rank one, that is, $|-|\simeq {}_R\mathbf{Mod}(R,-)$.   Let $(L,R)\in \mathit{TrHull}(|A|,*)$, with $x*y:=\mu(x\otimes y)$. Let $x,y,z\in |A|$ and let $\alpha\in R$. Then, $x*L(\alpha y+z)=R(x)*(\alpha y+z)=R(x)*(\alpha y)+R(x)*z=\alpha(R(x)*y)+R(x)*z=\alpha(x*L(y))+x*L(z)=x*(\alpha L(y)+L(z))$ so that $L(\alpha y+z)=\alpha L(y)+L(z)$ since $A$ is faithful. Similarly $R$ is also linear. Consequently, $(L,R)\in \mathit{Mult}_{\mathbb{Ab}}(\mathsf{A})$. 

\item According to~\cite{Wang} any commutative Banach algebra $\mathsf{B}=(B,*)$ without order (that is, $x*y=0$ for all $x$ implies $x=0$) and such that $B^2=B$ (that is, $*$ is onto), is concrete in $\mathbb{Ban}=(\mathbf{Ban},\hat{\otimes},\mathbb{C})$.  In fact it is proven in~\cite{Wang} that under the condition that $\mathsf{B}$ is without order and with $|-|$ the forgetful functor into $\mathbf{Set}$, $L\in \mathbf{Act}_{(|B|,*)}((|B|,*),(|B|,*))$ iff $(L,L)\in \mathit{TrHull}(|B|,*)$ iff $(L,L)\in \mathit{Mult}_{\mathbb{Ban}}(\mathsf{B})$. So if one assumes furthermore that $B^2=B$, that is, $*$ is epi, and thus so is the multiplication as a morphism $B\hat{\otimes}B\to B$, since faithful functors reflect epimorphisms, then $(L,R)\in \mathit{Mult}_{\mathbb{Ban}}(\mathsf{B})$ implies $L=R$ by Prop.~\ref{prop:for_com_sem_left_equal_right}.
\end{enumerate}
\end{example}

\section{Around  ordinary semigroups}\label{sec:around_ord_sem}

In this section one first turns the ordinary translational hull construction into a functor from a subcategory of semigroups to monoids, which behaves almost like a left adjoint to the forgetful functor from monoids to semigroups. Secondly, one extends the abstract multiplier monoid construction into a functor from a subcategory of $\mathbf{IntSem}$ to monoids, and the concretization homomorphisms then  become a natural transformation.

\subsection{The translational hull functor}\label{sec:around_ord_sem_1}

\begin{definition}\label{def:non-degeneracy_of_a_semigroup}
Let $\mathsf{S}=(S,*)$ be an ordinary semigroup. It is said to be 
\begin{enumerate}
\item {\em globally idempotent} when $S*S=S$, that is, $S=\{\, x*y\colon x,y\in S\,\}$ or in other terms $*\colon S\times S\to S$ is onto,
\item {\em right non-degenerate} (resp. {\em left non-degenerate}) when $x*y=x*z$ (resp. $y*x=z*x$) for all $x\in S$ implies $y=z$. It is {\em non-degenerate} when it is both left and right non-degenerate.
\end{enumerate}
\end{definition}
The first easy results below are stated without proofs.

\begin{lemma}\label{lem:non-deg}
$\mathsf{S}=(S,*)$ is right (resp. left) non-degenerate iff $\mathfrak{L}_{\mathsf{S}}\colon \mathsf{S}\to \mathsf{LTr}(\mathsf{S})$, $x\mapsto \mathfrak{L}_x$ (resp. $\mathfrak{R}_{\mathsf{S}}\colon \mathsf{S}\to \mathsf{RTr}(\mathsf{S})$, $x\mapsto \mathfrak{R}_x$) is one-to-one.  If $\mathsf{S}$ is non-degenerate, then $\mathfrak{M}_{\mathsf{S}}$ is one-to-one.
\end{lemma}

\begin{lemma}\label{lem:mon_is_automatically_glob_idem_and_left_right_non-degenerate}
Let $\mathsf{M}=(M,*,1)$ be a monoid. Then, $|\mathsf{M}|$ is globally idempotent and non-degenerate.
\end{lemma}

\begin{definition}
Let $S$ be a set and let $(T,*)$ be a  semigroup. 

\begin{enumerate}
\item Let $f\colon S\to T$ be a map. It is said to be {\em non-degenerate} if $T=\{\, f(s)*t\colon s\in S,\ t\in T\,\}=\{\, t*f(s)\colon s\in S,\ t\in T\,\}$. 
\item Let $f\colon S\to \mathit{TrHull}(T,*)$ be a map. Let $f_L:=p_1\circ f$ and $f_R:=p_2\circ f$ where $\mathit{LTr}(T,*)\xleftarrow{p_1}\mathit{TrHull}(T,*)\xrightarrow{p_2}\mathit{RTr}(T,*)$ is the canonical pullback diagram. In other words for each $s\in S$, $f(s)=(f_L(s),f_R(s))$.
\item Let $f\colon S\to \mathit{TrHull}(T,*)$ be a map. It is said to be {\em translation non-degenerate} when $T=\{\, f_L(s)(t)\colon s\in S,\ t\in T\,\}=\{\, f_R(s)(t)\colon s\in S,\ t\in T\,\}$. 
\end{enumerate}
\end{definition}

\begin{lemma}\label{lem:glob_idem_equals_id_S_non-degenerate}
Let $\mathsf{S}=(S,*)$ be a semigroup. $\mathsf{S}$ is globally idempotent iff $id_S$ is non-degenerate.
\end{lemma}


Let $f\colon S\to T$ be a map, and let $\mathsf{T}=(T,\cdot)$ be a semigroup. Noticing that $(\mathfrak{M}_{\mathsf{T}}\circ f)_L=\mathfrak{L}_{\mathsf{T}}\circ f$ and $(\mathfrak{M}_{\mathsf{T}}\circ f)_R=\mathfrak{R}_{\mathsf{T}}\circ f$, one obtains the following result and its corollary (also using Lemma~\ref{lem:glob_idem_equals_id_S_non-degenerate}). 
\begin{lemma}\label{lem:non-degeneracy2}
Let $S$ be a set and let $\mathsf{T}=(T,\cdot)$ be a semigroup. Let $f\colon S\to T$ be a map. $f$ is a non-degenerate map iff $\mathfrak{M}_{\mathsf{T}}\circ f\colon S\to |\mathsf{TrHull}(\mathsf{T})|$ is translation non-degenerate. 
\end{lemma}

\begin{corollary}\label{cor:non-degeneracy2}
Let $\mathsf{S}=(S,*)$ be a globally idempotent semigroup. Then, $\mathfrak{M}_{\mathsf{S}}\colon S\to |\mathsf{TrHull}(\mathsf{S})|$ is translation  non-degenerate.
\end{corollary}

Let us state for a later use the following two easy lemmas.
\begin{lemma}\label{lem:for_composition}
Let $\mathsf{S},\mathsf{T},\mathsf{U}$ be semigroups. Let $\mathsf{S}\xrightarrow{f}\mathsf{T}$ and $\mathsf{T}\xrightarrow{g}\mathsf{U}$ be non-degenerate homomorphism of semigroups. Then, so is $\mathsf{S}\xrightarrow{g\circ f}\mathsf{U}$.
\end{lemma}

\begin{lemma}\label{lem:hom_of_mons_are_non-degenerate}
Let $\mathsf{M}=(M,*,1)$ and $\mathsf{N}=(N,*,1)$ be monoids and let $f\colon \mathsf{M}\to \mathsf{N}$ be a homomorphism of monoids. Then, $|f|\colon (M,*)\to (N,*)$ is a non-degenerate homomorphism of semigroups.  
\end{lemma}

\begin{theorem}\label{thm:extension}
Let $\mathsf{S}=(S,*)$ and $\mathsf{T}=(T,\cdot)$ be semigroups, with $\mathsf{T}$  non-degenerate. Let $f\colon \mathsf{S}\to |\mathsf{TrHull}(\mathsf{T})|$ be a homomorphism of semigroups which is translation  non-degenerate. Then, there is a unique homomorphism of monoids $f^\sharp\colon \mathsf{TrHull}(\mathsf{S})\to \mathsf{TrHull}(\mathsf{T})$ such that $|f^\sharp|\circ \mathfrak{M}_{\mathsf{S}}=f$. Moreover $|f^\sharp|\colon |\mathsf{TrHull}(\mathsf{S})|\to |\mathsf{TrHull}(\mathsf{T})|$ is also translation  non-degenerate.
\end{theorem}

\begin{proof}
Let $s\in S$. Then, $\mathfrak{M}_{\mathsf{S}}(L(s))=(\mathfrak{L}_{\mathsf{S}}(L(s)),\mathfrak{R}_{\mathsf{S}}(L(s)))$. Let $s'\in S$. Then, $\mathfrak{L}_{\mathsf{S}}(L(s))(s')=L(s)*s'=L(s*s')=L(\mathfrak{L}_{\mathsf{S}}(s)(s'))$ so that $\mathfrak{L}_{\mathsf{S}}(L(s))=L\circ \mathfrak{L}_{\mathsf{S}}(s)$. Also $\mathfrak{R}_{\mathsf{S}}(L(s))(s')=s'*L(s)=R(s')*s=\mathfrak{R}_{\mathsf{S}}(s)(R(s'))$ so that $\mathfrak{R}_{\mathsf{S}}(L(s))=\mathfrak{R}_{\mathsf{S}}(s)\circ R$. Therefore, $\mathfrak{M}_{\mathsf{S}}(L(s))=(L\circ \mathfrak{L}_{\mathsf{S}}(s),\mathfrak{R}_{\mathsf{S}}(s)\circ R)=(L,R)\star \mathfrak{M}_{\mathsf{S}}(s)$. Similarly, $\mathfrak{M}_{\mathsf{S}}(R(s))=(L,R)\star \mathfrak{M}_{\mathsf{S}}(s)$.

Assume that $f^\sharp$ exists. Then, for each $(L,R)\in \mathit{TrHull}(\mathsf{S})$, $s\in S$ and $t\in T$, $f_L(L(s))(t)=f^\sharp_L(\mathfrak{M}_{\mathsf{S}}(L(s)))(t)=f^\sharp_L((L,R)\circ \mathfrak{M}_{\mathsf{S}}(s))(t)=f^\sharp_L(L,R)(f^\sharp_L(\mathfrak{M}_{\mathsf{S}}(s))(t))=f^\sharp_L(L,R)(f_L(s)(t))$. In a way similar,  $f_R(R(s))(t)=f^\sharp_R(L,R)(f_R(s)(t))$.  Since $f$ is non-degenerate, it follows that $f^\sharp$ is unique. 

Now let $(L,R)\in  \mathit{TrHull}(\mathsf{S})$. Let $u\in T$. Then there are $s\in S$ and $t\in T$ such that $u=f_L(s)(t)$. Define $f^\sharp_L(L,R)(u):=f_L(L(s))(t)$. There are also $r\in S$ and $v\in T$ such that $u=f_R(r)(v)$, and one also defines $f^\sharp_R(L,R)(u):=f_R(R(r))(v)$. 

To be well-defined one has to check that for each $(L,R)\in  \mathit{TrHull}(\mathsf{S})$ if $f_L(s)(t)=f_L(s')(t')$, then $f_L(L(s))(t)=f_L(L(s'))(t')$ and for each $(L,R)\in  \mathit{TrHull}(\mathsf{S})$ if $f_R(s)(t)=f_R(s')(t')$, then $f_R(R(s))(t)=f_R(R(s'))(t')$.

Now let $(L,R)\in  \mathit{TrHull}(\mathsf{S})$.  Assume that $f_R(s)(t)=f_R(s')(t')$. Let $u\in S$, $v\in T$. Then, 
\begin{equation}
\begin{array}{lll}
f_R(R(s))(t)\cdot f_L(u)(v)&=&f_R(u)(f_R(R(s)))(t)\cdot v\\
&=&f_R(R(s)*u)(t)\cdot v\\
&=&f_R(s*L(u))(t)\cdot v\\
&=&f_R(L(u))(f_R(s)(t))\cdot v\\
&=&f_R(L(u))(f_R(s')(t'))\cdot v\\
&=&f_R(s'*L(u))(t')\cdot v\\
&=&f_R(R(s')*u)(t')\cdot v\\
&=&f_R(u)(f_R(R(s'))(t'))\cdot v\\
&=&f_R(R(s'))(t')\cdot f_L(u)(v).
\end{array}
\end{equation}
Since $f$ is non-degenerate, this is equivalent to $f_R(R(s))(t)\cdot w=f_R(R(s'))(t')\cdot w$ for each $w\in T$, and since $\mathsf{T}$ is both left and right non-degenerate, $f_R(R(s))(t)=f_R(R(s'))(t')$. Consequently, $f^\sharp$ is well-defined. 

One observes that $f^\sharp_L(\mathfrak{M}_{\mathsf{S}}(a))(u)=f_L^\sharp(\mathfrak{M}_{\mathsf{S}}(a))(f_L(s)(t))=f_L(\mathfrak{L}_{\mathsf{S}}(a)(s))(t)=f_L(a*s)(t)=f_L(a)(f_L(s)(t))=f_L(a)(u)$ when $u=f_L(s)(t)$, that is, $f^\sharp_L\circ \mathfrak{M}_{\mathsf{S}}=f_L$ and also $f^\sharp_R(\mathfrak{M}_{\mathsf{S}}(a))(u)=f_R(a)(u)$ when $u=f_R(s)(t)$, that is, $f^\sharp_R\circ \mathfrak{M}_{\mathcal{S}}=f_R$. Consequently, $f^\sharp\circ \mathfrak{M}_{\mathsf{S}}=f$.

Now let us check that for each $(L,R)\in  \mathit{TrHull}(\mathsf{S})$, $(f^\sharp_L(L,R),f^\sharp_R(L,R))\in  \mathit{TrHull}(\mathsf{T})$. One has $f^\sharp_L(L,R)(u\cdot v)=f^\sharp_L(L,R)(f_L(s)(t)\cdot v)=f^\sharp(L,R)(f_L(s)(t\cdot v))=f_L(L(s))(t\cdot v)=f_L(L(s))(t)\cdot v=f^\sharp_L(L,R)(u)\cdot v$ when $u=f_L(s)(t)$, and similarly $f^\sharp_R(L,R)(u\cdot v)=u\cdot f^\sharp_R(L,R)(v)$ when $v=f_R(s)(t)$. Let $u=f_R(s)(t)$ and $u'=f_L(s')(t')$. Then, 
\begin{equation}
\begin{array}{lll}
f^\sharp_R(L,R)(u)\cdot u'&=&f_R(R(s))(t)\cdot u'\\
&=&t\cdot f_L(R(s))(u')\\
&=&t\cdot f_L(R(s))(f_L(s')(t'))\\
&=&t\cdot f_L(R(s)*s')(t')\\
&=&t\cdot f_L(s*L(s'))(t')\\
&=&t\cdot f_L(s)(f_L(L(s')(t')))\\
&=&t\cdot f_L(s)(f^\sharp_L(L,R)(u'))\\
&=&f_R(s)(t)\cdot f^\sharp_L(L,R)(u')\\
&=&u\cdot f^\sharp_L(L,R)(u').
\end{array}
\end{equation}
Finally let us check that $f^\sharp$ is a homomorphism: let $(L,R),(L',R')\in  \mathit{TrHull}(\mathsf{S})$. Let $s\in S$ and $t\in T$. Then,

\begin{equation}
\begin{array}{lll}
f^\sharp_L(L'\circ L,R\circ R')(f_L(s)(t))&=&f_L(L'(L(s)))(t)\\
&=&f^\sharp_L(L',R')(f_L(L(s))(t))\\
&=&f^\sharp_L(L',R')(f^\sharp_L(L,R)(f_L(s)(t)).
\end{array}
\end{equation}
Likewise one also has $f^\sharp_R(L'\circ L,R\circ R')(f_R(s)(t))=f^\sharp_R(L,R)(f^\sharp_R(L',R')(f_R(s)(t))$. Consequently, $f^\sharp$ is a homomorphism of semigroups. Since it evidently preserves the units, it is in fact a homomorphism of monoids. 

Let $u\in T$. Since $f$ is translation  non-degenerate, there are $s\in S$ and $t\in T$ such that $f_L(s)(t)=u$. Then, $f^\sharp_L(\mathfrak{M}_{\mathsf{S}}(s))(t)=f_L(s)(t)=u$.  The last assertion now follows easily. 
\end{proof} 

By defining $\mathit{TrHull}(f):=(\mathfrak{M}_{\mathsf{T}}\circ f)^{\sharp}$ one obtains the 
\begin{corollary}\label{cor:extension}
Let $\mathsf{S}$ and $\mathsf{T}$ be semigroups, with $\mathsf{T}$ non-degenerate. Let $f\colon \mathsf{S}\to \mathsf{T}$ be a homomorphism of semigroups which is  non-degenerate. Then, there is a unique homomorphism of monoids $\mathit{TrHull}(f)\colon \mathsf{TrHull}(\mathsf{S})\to \mathsf{TrHull}(\mathsf{T})$ such that the following diagram commutes in $\mathbf{Sem}$. Moreover $|\mathit{TrHull}(f)|$ is translation  non-degenerate.
\begin{equation}
\xymatrix{
|\mathsf{TrHull}(\mathsf{S})|)\ar[r]^{|\mathit{TrHull}(f)|}&|\mathsf{TrHull}(\mathsf{T})|\\
\ar[u]^{\mathfrak{M}_{\mathsf{S}}}\mathsf{S}\ar[r]_{f}&\mathsf{T}\ar[u]_{\mathfrak{M}_{\mathsf{T}}}
}
\end{equation}
\end{corollary}

Let $\mathbf{Sem}_{nd}$ be the subcategory of semigroups whose objects are the globally idempotent and non-degenerate semigroups, and $\mathbf{Sem}_{nd}(\mathsf{S},\mathsf{T})$ consisting of all homomorphisms of semigroups which are non-degenerate. (From Lemma~\ref{lem:glob_idem_equals_id_S_non-degenerate} it follows that $id_{\mathsf{S}}\in \mathbf{Sem}_{nd}(\mathsf{S},\mathsf{S})$ and from Lemma~\ref{lem:for_composition} that composition of non-degenerate homomorphisms is a non-degenerate homomorphism.)

\begin{corollary}\label{cor:TrHull_functor}
$\mathit{TrHull}$ provides a functor from $\mathbf{Sem}_{nd}$ to $\mathbf{Mon}$, and $\mathfrak{M}=(\mathfrak{M}_{\mathsf{S}})_{\mathsf{S}}\colon J\Rightarrow |-|\circ \mathit{TrHull}\colon \mathbf{Sem}_{nd}\to \mathbf{Sem}$, where $J\colon \mathbf{Sem}_{nd}\to \mathbf{Sem}$ is the canonical (non-full) embedding functor.
\end{corollary}

According to Lemmas~\ref{lem:mon_is_automatically_glob_idem_and_left_right_non-degenerate} and~\ref{lem:hom_of_mons_are_non-degenerate},  the forgetful functor $|-|\colon \mathbf{Mon}\to \mathbf{Sem}$ co-restricts to a functor still denoted $|-|$ from $\mathbf{Mon}$ to $\mathbf{Sem}_{nd}$. Note that $|-|\colon \mathbf{Mon}\to \mathbf{Sem}_{nd}$ is still faithful because so is the embedding functor $\mathbf{Sem}_{nd}\hookrightarrow \mathbf{Sem}$.

Nevertheless $\mathfrak{M}$ is not a natural transformation $id_{\mathbf{Sem}_{nd}}\Rightarrow |-|\circ \mathit{TrHull}$ from $\mathbf{Sem}_{nd}$ to $\mathbf{Sem}_{nd}$ because $\mathfrak{M}_{\mathsf{S}}$ while translation non-degenerate,  may fail to be non-degenerate. Indeed, let $s\in S$ and let $(L,R)\in \mathit{TrHull}(\mathsf{S})$. Then, $\mathfrak{M}_{\mathsf{S}}(s)\star (L,R)=\mathfrak{M}_{\mathsf{S}}(R(s))$. So if it happens that $\mathit{TrHull}(\mathsf{S})$ contains other multipliers than the inner ones, $\mathfrak{M}_{\mathsf{S}}$ cannot be non-degenerate. 

Nevertheless $\mathit{TrHull}\colon \mathbf{Sem}_{nd}\to \mathbf{Mon}$ is very close to be a left adjoint to $|-|\colon \mathbf{Mon}\to \mathbf{Sem}_{nd}$ as is shown by the next theorem (the only missing part is that $\mathfrak{M}$  {\em cannot} be a unit for this ``adjunction''), and in any case, $\mathit{TrHull}$ may be thought to as a {\em unitarization} functor.
\begin{theorem}\label{thm:close_to_left_adjoint}
Let $\mathsf{S}=(S,*)$ be a globally idempotent and non-degenerate semigroup and let $\mathsf{M}=(M,*,1)$ be a monoid. Let $f\colon \mathsf{S}\to |\mathsf{M}|$ be a non-degenerate homomorphism of semigroups. Then, there is a unique homomorphism of monoids $\mathsf{Mult}(\mathsf{S})\xrightarrow{f^{\flat}}\mathsf{M}$ such that $|f^{\flat}|\circ \mathfrak{M}_{\mathsf{S}}=f$. 
\end{theorem}

\begin{proof}
The homomorphism of semigroups $\mathsf{S}=(S,*)\xrightarrow{f}|\mathsf{M}|=(M,*)\xrightarrow{\mathfrak{M}_{|\mathsf{M}|}}|\mathsf{TrHull}(|\mathsf{M}|)|$ is  translation  non-degenerate by Lemma~ \ref{lem:non-degeneracy2} since $f$ is non-degenerate, and since $(M,*)=|\mathsf{M}|$ is non-degenerate by Lemma~\ref{lem:mon_is_automatically_glob_idem_and_left_right_non-degenerate}, Theorem~\ref{thm:extension} tells us that there is a unique homomorphism of monoids $g_f:=(\mathfrak{M}_{|\mathsf{M}|}\circ f)^\sharp\colon \mathsf{TrHull}(\mathsf{S})\to \mathsf{TrHull}(|\mathsf{M}|)$ such that $|g_f|\circ \mathfrak{M}_{\mathsf{S}}=\mathfrak{M}_{|\mathsf{M}|}\circ f$. According to Lemma~\ref{lem:can_hom_is_iso_for_mon}, $\mathfrak{M}_{\mathsf{M}}\colon \mathsf{M}\to \mathsf{TrHull}(|\mathsf{M}|)$ is an isomorphism, and thus one may consider the homomorphism of monoids $f^\flat:=\mathfrak{M}_{\mathsf{M}}^{-1}\circ g_f\colon \mathsf{TrHull}(\mathsf{S})\to \mathsf{M}$. It satisfies $|f^\flat|\circ \mathfrak{M}_{\mathsf{S}}=\mathfrak{M}_{|\mathsf{M}|}^{-1}\circ \mathfrak{M}_{|\mathsf{M}|}\circ f=f$. Uniqueness is easily checked.
\end{proof}

One may also use $\mathit{TrHull}$ to define another category  in order to turn the translational hull construction into an adjoint. Let $\mathsf{S},\mathsf{T},\mathsf{U}$ be semigroups, where $\mathsf{T},\mathsf{U}$ are non-degenerate. Let $f\colon \mathsf{S}\to |\mathsf{TrHull}(\mathsf{T})|$ and $g\colon \mathsf{T}\to |\mathsf{TrHull}(\mathsf{U})|$ be translation  non-degenerate homomorphisms of semigroups. Then, $|g^\sharp|\circ f\colon \mathsf{S}\to |\mathit{TrHull}(\mathsf{U})|$ is of course a homomorphism of semigroups. Let us check that it is translation  non-degenerate too. One first notices that $(g^\sharp\circ f)(s)=g^\sharp(f_L(s),f_R(s))=(g^\sharp_L(f_L(s),f_R(s)),g^\sharp_R(f_L(s),f_R(s)))$, $s\in S$. Consequently, $(g^\sharp\circ f)_L=g^\sharp_L\circ f$ and $(g^\sharp\circ f)_R=g^\sharp_R\circ f$. Now let $u\in U$. One has to show that there are $s\in S$ and $v\in U$ such that 
$g^\sharp_L(f_L(s),f_R(s))(v)=u$. For this $u\in U$ there are $t\in T$ and $v\in U$ such that $g_L(t)(v)=u$. Moreover there are $s\in S$ and $r\in T$. Then, $f_L(s)(r)=t$. Whence $g_L(f_L(s)(r))(v)=u$. By definition of $g^\sharp$, $g^\sharp(f(s))(g_L(r)(v))=g_L(f_L(s)(r))(v)=g_L(t)(v)=u$. In a same way  given $u\in U$, there are $s\in S$ and $v\in U$ such that $g^\sharp_R(f_L(s),f_R(s))(v)=u$.  

Now let $\mathsf{S}$ be a globally idempotent semigroup. Then, according to Corollary~\ref{cor:non-degeneracy2}, $\mathfrak{M}_{\mathsf{S}}\colon \mathsf{S}\to |\mathsf{TrHull}(\mathsf{S})|$ is translation  non-degenerate. 

By the above one can define the category $\mathbf{MultSem}_{nd}$  whose objects are globally idempotent, non-degenerate semigroups, $\mathbf{MultSem}_{nd}(\mathsf{S},\mathsf{T})$ consists of the homomorphisms of semigroups $\mathsf{S}\to |\mathsf{TrHull}(\mathsf{T})|$ which are translation  non-degenerate, $g\bullet f:=|g^\sharp|\circ f$, and $id_{\mathsf{S}}:=\mathfrak{M}_{\mathsf{S}}$. (Note that $\mathbf{MultSem}_{nd}$ is {\em not} a subcategory of $\mathbf{Sem}$.) The morphisms of $\mathbf{MultSem}_{nd}$ are similar to that used in~\cite{Daele} to define multiplier Hopf algebras.  

According to Lemma~\ref{lem:non-degeneracy2} one has a functor $E\colon \mathbf{Sem}_{nd}\to \mathbf{MultSem}_{nd}$, $(\mathsf{S}\xrightarrow{f}\mathsf{T})\mapsto (\mathsf{S}\xrightarrow{\mathfrak{M}_{\mathsf{T}}\circ f}\mathsf{T})$. Since for each non-degenerate semigroup $\mathsf{S}$, $\mathfrak{M}_{\mathsf{S}}$ is one-to-one it follows that the functor $E$ is faithful and thus provides an embedding of $\mathbf{Sem}_{nd}$ into $\mathbf{MultSem}_{nd}$.

Now one may also define a functor $\widetilde{\mathit{TrHull}}\colon \mathbf{MultSem}_{nd}\to \mathbf{Mon}$ as follows: $\widetilde{\mathit{TrHull}}(\mathsf{S}\xrightarrow{f}|\mathsf{TrHull}(\mathsf{T})|):=\mathsf{TrHull}(\mathsf{S})\xrightarrow{f^\sharp}\mathsf{TrHull}(\mathsf{T})$ and one recovers $\mathit{TrHull}\colon \mathbf{Sem}_{nd}\to \mathbf{Mon}$ as  $\mathbf{Sem}_{nd}\xrightarrow{E} \mathbf{MultSem}_{nd}\xrightarrow{\widetilde{\mathit{TrHull}}}\mathbf{Mon}$.  

Let $\mathsf{M}=(M,*,1)$, and $\mathsf{N}=(N,*,1)$ be monoids and let $\mathsf{M}\xrightarrow{f}\mathsf{N}$ be a homomorphism of monoids. Then, $|f|\colon |\mathsf{M}|=(M,*)\to |\mathsf{N}|=(N,*)$ is non-degenerate by Lemma~\ref{lem:hom_of_mons_are_non-degenerate}, and thus $\mathfrak{M}_{|\mathsf{N}|}\circ |f|\colon |\mathsf{M}|\to |Mult(|\mathsf{M}|)|$ is translation  non-degenerate (Lemma~\ref{lem:non-degeneracy2}). Since any monoid is automatically globally idempotent and non-degenerate, it is clear that $(\mathfrak{M}_{|\mathsf{N}|}\circ |f|)$ may be considered as a $\mathbf{MultSem}_{nd}$-morphism from $|\mathsf{M}|$ to $|\mathsf{N}|$. So one may consider the unique homomorphism of monoids $(\mathfrak{M}_{|\mathsf{N}|}\circ |f|)^\sharp\colon \mathsf{TrHull}(|\mathsf{M}|)\to \mathsf{TrHull}(|\mathsf{N}|)$ such that $|(\mathfrak{M}_{|\mathsf{N}|}\circ |f|)^\sharp|\circ \mathfrak{M}_{|\mathsf{M}|}=\mathfrak{M}_{|\mathsf{N}|}\circ |f|$, that is, $(\mathfrak{M}_{|\mathsf{N}|}\circ |f|)^\sharp=\widetilde{\mathit{TrHull}}(\mathfrak{M}_{|\mathsf{N}|}\circ |f|)=\mathit{TrHull}(|f|)$. 
Now let $\mathsf{P}$ be also a monoid and let $g\colon \mathsf{N}\to\mathsf{P}$ be a homomorphism of monoids. Then, $(\mathfrak{M}_{|\mathsf{P}|}\circ |g|)\bullet (\mathfrak{M}_{|\mathsf{N}|}\circ |f|)=|(\mathfrak{M}_{|\mathsf{P}|}\circ |g|)^{\sharp}|\circ (\mathfrak{M}_{|\mathsf{N}|}\circ |f|)=\mathfrak{M}_{|\mathsf{P}|}\circ |g|\circ |f|=\mathfrak{M}_{|\mathsf{P}|}\circ |g\circ f|$.

All of this provides a functor $F\colon \mathbf{Mon}\to \mathbf{MultSem}_{nd}$ given by $F(\mathsf{M}\xrightarrow{f}\mathsf{N}):=
|\mathsf{M}|\xrightarrow{|f|}|\mathsf{N}|\xrightarrow{\mathfrak{M}_{|\mathsf{N}|}}|\mathsf{TrHull}(|\mathsf{N}|)|$. One has the factorisation $F=\mathbf{Mon}\xrightarrow{|-|}\mathbf{Sem}_{nd}\xrightarrow{E}\mathbf{MultSem}_{nd}$. 

\begin{proposition}\label{prop:adjunction_F_tilde_M}
$F$ is a left adjoint of $\widetilde{\mathit{TrHull}}$.
\end{proposition}
\begin{proof}
Let $\mathsf{M}=(M,*,1)$ be a monoid.  Let $\mathsf{S}$ be a globally idempotent and non-degenerate semigroup and let $\mathsf{M}\xrightarrow{f}\mathsf{TrHull}(\mathsf{S})$ be a homomorphism of monoids. Let us consider the homomorphism of monoids $\tilde{f}:=\mathsf{TrHull}(|\mathsf{M}|)\xrightarrow{\mathfrak{M}_{\mathsf{M}}^{-1}}\mathsf{M}\xrightarrow{f}\mathsf{TrHull}(\mathsf{S})$. Then, $|\tilde{f}|\in \mathbf{Sem}(|\mathsf{TrHull}(|\mathsf{M}|)|,|\mathsf{TrHull}(\mathsf{S})|)=\mathbf{MultSem}_{nd}(|\mathsf{TrHull}(|\mathsf{M}|)|,\mathsf{S})$ is such that $|\tilde{f}|\circ \mathfrak{M}_{|\mathsf{M}|}=|f|$. Therefore, as $|f|$ is non-degenerate, $\tilde{f}=|f|^\sharp=\widetilde{\mathit{TrHull}}(|f|)$. So by construction $|\widetilde{\mathit{TrHull}}(|f|)|\circ |\mathfrak{M}_{\mathsf{M}}|=|\widetilde{\mathit{TrHull}}(|f|)|\circ \mathfrak{M}_{|\mathsf{M}|}=|f|$ as homomorphisms of semigroups and thus $\widetilde{\mathit{TrHull}}(|f|)\circ \mathfrak{M}_{\mathsf{M}}=f$ as homomorphisms of monoids since the forgetful functor is faithful. Now let $\phi\in \mathbf{MultSem}_{nd}(|\mathsf{TrHull}(|\mathsf{M}|)|,\mathsf{S})$  such that $\widetilde{\mathit{TrHull}}(\phi)\circ \mathfrak{M}_{\mathsf{M}}=f$. Then $\widetilde{\mathit{TrHull}}(\phi)=f\circ \mathfrak{M}^{-1}_{\mathsf{M}}=|f|^\sharp$ and thus $\phi=|f|$. Consequently, $|f|\in \mathbf{MultSem}_{nd}(|\mathsf{M}|,\mathsf{S})$ is the unique morphism such that $\widetilde{\mathit{TrHull}}(|f|)\circ \mathfrak{M}_{\mathsf{M}}=f$. By direct inspection one checks that $\mathfrak{M}:=(\mathfrak{M}_{\mathsf{M}})_{\mathsf{M}}\colon id_{\mathbf{Mon}}\Rightarrow \widetilde{\mathit{TrHull}}\circ F\colon \mathbf{Mon}\to \mathbf{Mon}$. 
\end{proof}

%
%

\subsection{The multiplier functor}\label{sec:around_ord_sem_2}

Let $\mathbb{C}=(C,\otimes_C,I)$ be a monoidal category. Let $\mathsf{S}=(S,\mu)$ be a  semigroup  in $\mathbb{C}$ and let $X$ be a set. For a map $X\xrightarrow{f}\mathit{Mult}_{\mathbb{C}}(\mathsf{S})$, one defines $f_L\colon X\to \mathbf{Act}_{\mathsf{S}}(\mathbb{C})((S,\mu),(S,\mu))$ and $f_R\colon X\to {}_{\mathsf{S}}\mathbf{Act}(\mathbb{C})((S,\mu),(S,\mu))$ by $f_L:=p_1\circ f$ and $f_R:=p_2\circ f$, where $\mathbf{Act}_{\mathsf{S}}(\mathbb{C})((S,\mu),(S,\mu))\xleftarrow{p_1}\mathit{Mult}_{\mathbb{C}}(\mathsf{S})\xrightarrow{p_2}{}_{\mathsf{S}}\mathbf{Act}(\mathbb{C})((S,\mu),(S,\mu))$ is the canonical pullback diagram. 

A map $X\xrightarrow{f}\mathit{Mult}_{\mathbb{C}}(\mathsf{S})$ is said to be {\em multiplier non-degenerate} when $C(I,S)=\{\, f_L(x)\circ s\colon x\in X,\ s\in C(I,S)\,\}=\{\, f_R(x)\circ s\colon x\in X,\ s\in C(I,S)\,\}$. 

\begin{lemma}\label{lem:non-degen_to_mult_non-degen_2}
Let $f\colon X\to C(I,S)$ be a map. $f$ is non-degenerate iff $X\xrightarrow{f} C(I,S)\xrightarrow{\mathcal{M}_{\mathbb{C},\mathsf{S}}} \mathit{Mult}_{\mathbb{C}}(\mathsf{S})$ is multiplier non-degenerate.
\end{lemma}
\begin{proof}
One first notes that $(\mathcal{M}_{\mathbb{C},\mathsf{S}}\circ f)_L=\mathcal{L}_{\mathbb{C},\mathsf{S}}\circ f$ and $(\mathcal{M}_{\mathbb{C},\mathsf{S}}\circ f)_R=\mathcal{R}_{\mathbb{C},\mathsf{S}}\circ f$. Since for each $x\in X$ and $s\in C(I,S)$, $f(x)\bullet_\mu s=L_{f(x)}\circ s$ and $s\bullet_\mu R_{f(x)}=f(x)\circ s$, the expected result follows easily.
\end{proof}


\begin{theorem}\label{thm:extension_for_internal_multipliers}
Let $\mathsf{S}=(S,\mu_S)$ be a semigroup in $\mathbb{C}=(C,\otimes_C,I)$ and let $\mathsf{T}=(T,\mu_T)$ be a semigroup in $\mathbb{D}=(D,\otimes_D,J)$. Let us assume that  the convolution semigroup $\mathsf{Conv}_{J}(\mathsf{T})$ is  non-degenerate, and  $\mathsf{T}$ is concrete. 

Let $f\colon \mathsf{Conv}_I(\mathsf{S})\to (\mathit{Mult}_{\mathbb{D}}(\mathsf{T}),\star)$ be a homomorphism of semigroups which is translation non-degenerate. Then, there is a unique homomorphism of monoids $f^{M}\colon\mathsf{Mult}_{\mathbb{C}}(\mathsf{S}) \to \mathsf{Mult}_{\mathbb{D}}(\mathsf{T})$ such that $|f^M|\circ \mathcal{M}_{\mathbb{C},\mathsf{S}}=f$. {Moreover, $|Conc_{\mathbb{D}}|\circ f\colon \mathsf{Conv}_I(\mathsf{S})\to |\mathsf{TrHull}(\mathsf{Conv}_J(\mathsf{D}))|$ is multiplier non-degenerate and  $(|\mathit{Conc}_{\mathbb{D}}(\mathsf{T})|\circ f)^\sharp\circ \mathit{Conc}_{\mathbb{C}}(\mathsf{S})=\mathit{Conc}_{\mathbb{D}}(\mathsf{T})\circ f^M$.}
\end{theorem}

\begin{proof}
Preambule: let $(L,R)\in \mathit{Mult}_{\mathbb{C}}(\mathsf{S})$ and let $s\in C(I,S)$. Then, $\mathcal{M}_{\mathbb{C},\mathsf{S}}(L\circ s)=(L\circ L_s,R_s\circ R)$ because $L_{L\circ s}=L\circ L_s$ and $R_{L\circ s}=R_s\circ R$. Therefore $\mathcal{M}_{\mathbb{C},\mathsf{S}}(L\circ s)=(L,R)\star\mathcal{M}_{\mathbb{C},\mathsf{S}}(s)$.  Also, $\mathcal{M}_{\mathbb{C},\mathsf{S}}(R\circ s)=(L_s\circ L,R\circ R_s)=\mathcal{M}_{\mathbb{C},\mathsf{S}}(s)\star (L,R)$.

Now let us assume that the homomorphism $f^{M}\colon \mathsf{Mult}_{\mathbb{C}}(\mathsf{S})\to \mathsf{Mult}_{\mathbb{D}}(\mathsf{T})$ such that $|f^M|\circ \mathcal{M}_{\mathbb{C},\mathsf{S}}=f$ exists. Let $(L,R)\in \mathit{Mult}_{\mathbb{C}}(\mathsf{S})$ and $s\in C(I,S)$. Then, 
\begin{equation}
\begin{array}{lll}
f_L(L\circ s)&=&f^M_L(\mathcal{M}_{\mathbb{C},\mathsf{S}}(L\circ s))\\
&=&f^M_L((L,R)\star\mathcal{M}_{\mathbb{C},\mathsf{S}}(s))\\
&=&f^M_L(L,R)\circ f^M_L(\mathcal{M}_{\mathbb{C},\mathsf{S}}(s))\\
&=&f^M_{L}(L,R)\circ f_L(s).
\end{array}
\end{equation}
Now let $t\in D(J,T)$. Then, $f_L(L\circ s)\circ t=(f^M_{L}(L,R)\circ f_L(s))\circ t=f^M_{L}(L,R)\circ (f_L(s)\circ t)$. So if $g\colon  \mathsf{Mult}_{\mathbb{C}}(\mathsf{S})\to \mathsf{Mult}_{\mathbb{D}}(\mathsf{T})$ is a homomorphism such that $|f^M|\circ \mathcal{M}_{\mathbb{C},\mathsf{S}}=f$, then for each $t\in D(J,T)$, each $s\in C(I,S)$, $g_L(L,R)\circ (f_L(s)\circ t)=f^M_L(L,R)\circ (f_L(s)\circ t)$. Since $f$ is assumed non-degenerate, this is equivalent to the fact that for each $u\in D(J,T)$, $g_L(L,R)\circ u=f^M_L(L,R)\circ u$. Since $\mathsf{Conv}_J(\mathsf{T})$ is assumed non-degenerate, it follows that $g_L(L,R)=f^M_L(L,R)$ which guarantees uniqueness of $f^M_L$. Similarly, one obtains uniqueness of $f^M_R$, and so that  of $f^M$.

Let $(L,R)\in \mathit{Mult}_{\mathbb{C}}(\mathsf{S})$ and let $u\in D(J,T)$. Define $G_f(L,R)(u):=f_L(L\circ s)\circ t$ when $u=f_L(s)\circ t$, and define $D_f(L,R)(u):=f_R(R\circ s)\circ t$ when $u=f_R(s)\circ t$. 

Let us check that both $G_f(L,R)(u)$ and $D_f(L,R)(u)$ are well-defined. So let $s,s'\in C(I,S)$ and $t,t'\in D(J,T)$ such that $f_L(s)\circ t=f_L(s')\circ t'$. Let $u\in C(I,S)$ and $v\in D(J,T)$. Then,
\begin{equation}
\begin{array}{lll}
(f_R(u)\circ v)\bullet_{\mu_T} (f_L(L\circ s)\circ t)&=&D(J,f_R(u))(v)\bullet_{\mu_T}(f_L(L\circ s)\circ t)\\
&=&v\bullet_{\mu_T} D(J,f_L(u))(f_L(L\circ s)\circ t)\\
&=&v\bullet_{\mu_T} (f_L(u)\circ f_L(L\circ s)\circ t)\\
&=&v\bullet_{\mu_T} (f_L(u\bullet_{\mu_S} (L\circ s))\circ t)\\
&=&v\bullet_{\mu_T} (f_L(u\bullet_{\mu_S} C(I,L)(s))\circ t)\\
&=&v\bullet_{\mu_T} (f_L(C(I,R)(u)\bullet_{\mu_S} s)\circ t)\\
&=&v\bullet_{\mu_T} (f_L((R\circ u)\bullet_{\mu_S} s)\circ t)\\
&=&v\bullet_{\mu_T} (f_L(R\circ u)\circ f_L(s)\circ t)\\
&=&v\bullet_{\mu_T} (f_L(R\circ u)\circ f_L(s')\circ t')\\
&=&v\bullet_{\mu_T} (f_L((R\circ u)\bullet_{\mu_S} s')\circ t')\\
&=&v\bullet_{\mu_T} (f_L(C(I,R)(u)\bullet_{\mu_S} s')\circ t')\\
&=&v\bullet_{\mu_T} (f_L(u\bullet_{\mu_S} C(I,L)(s'))\circ t')\\
&=&v\bullet_{\mu_T} (f_L(u\bullet_{\mu_S} (L\circ s'))\circ t')\\
&=&v\bullet_{\mu_T} (f_L(u)\circ f_L(L\circ s')\circ t')\\
&=&v\bullet_{\mu_T} (D(J,f_L(u))(f_L(L\circ s')\circ t'))\\
&=&D(J,f_R(u))(v)\bullet_{\mu_T} (f_L(L\circ s')\circ t')\\
&=&(f_R(u)\circ v)\bullet_{\mu_T} (f_L(L\circ s')\circ t').
\end{array}
\end{equation}
Since $f$ is non-degenerate, this means that for each $w\in D(J,T)$, $w\bullet_{\mu_T} (f_L(L\circ s)\circ t)=w\bullet_{\mu_T} (f_L(L\circ s')\circ t')$ and since $\mathsf{Conv}_J(\mathsf{D})$ is non-degenerate, $f_L(L\circ s)\circ t=f_L(L\circ s')\circ t')$, and $G_f(L,R)(u)$ is well-defined. Likewise $D_f(L,R)(u)$ is well-defined. 

Now let us show that $(G_f(L,R),D_f(L,R))\in \mathit{TrHull}(\mathsf{Conv}_J(\mathsf{T}))$. Let $u,v\in D(J,T)$. Let $s\in C(I,S)$ and $t\in D(J,T)$ such that $u=f_L(s)\circ t$. Then, 
\begin{equation}
\begin{array}{lll}
G_f(L,R)(u\bullet_{\mu_{\mathsf{T}}} v)&=&G(L,R)((f_L(s)\circ t)\bullet_{\mathsf{T}} v)\\
&=&G_f(L,R)(D(J,f_L(s))(t)\bullet_{\mu_{\mathsf{T}}} v)\\
&=&G_f(L,R)(D(J,f_L(s))(t\bullet_{\mu_{\mathsf{T}}} v))\\
&=&
G_f(L,R)(f_L(s)\circ (t\bullet_{\mu_{\mathsf{T}}} v))\\
&=&f_L(L\circ s)\circ (t\bullet_{\mu_{\mathsf{T}}} v)\\
&=&D(J,f_L(L\circ s)(t\bullet_{\mu_{\mathsf{T}}} v)\\
&=&
D(J,f_L(L\circ s))(t)\bullet_{\mu_{\mathsf{T}}} v\\
&=&(f_L(L\circ s)\circ t)\bullet_{\mu_{\mathsf{T}}}v\\
&=&G_f(L,R)(u)\bullet_{\mu_{\mathsf{T}}} v.
\end{array}
\end{equation}
Likewise $D_f(L,R)(u\bullet_{\mu_{\mathsf{T}}} v)=u\bullet_{\mu_{\mathsf{T}}} D_f(L,R)(v)$, $u,v\in D(J,T)$.

Let $u,u'\in D(J,T)$ and let $s,s'\in C(I,S)$, $t,t'\in D(J,T)$ such that $u=f_R(s)\circ t$ and $u'=f_L(s')\circ t'$. Then, 
\begin{equation}
\begin{array}{lll}
D(L,R)(u)\bullet_{\mu_{\mathsf{T}}} u'&=&(f_R(R\circ s)\circ t)\bullet_{\mu_{\mathsf{T}}} u'\\
&=&D(J,f_R(R\circ s))(t)\bullet_{\mu_{\mathsf{T}}}u'\\
&=&t\bullet_{\mu_{\mathsf{T}}}D(J,f_L(R\circ s))(u')\\
&=&t\bullet_{\mu_{\mathsf{T}}} (f_L(R\circ s)\circ u')\\
&=&t\bullet_{\mu_{\mathsf{T}}} (f_L(R\circ s)\circ f_L(s')\circ t')\\
&=&t\bullet_{\mu_{\mathsf{T}}} (f_L((R\circ s)\bullet_{\mu_{\mathsf{S}}} s')\circ t')\\
&=&t\bullet_{\mu_{\mathsf{T}}} (f_L(C(I,R)(s)\bullet_{\mu_{\mathsf{S}}}s')\circ t')\\
&=&t\bullet_{\mu_{\mathsf{T}}} (f_L(s\bullet_{\mu_{\mathsf{S}}} C(I,L)(s'))\circ t')\\
&=&t\bullet_{\mu_{\mathsf{T}}} (f_L(s\bullet_{\mu_{\mathsf{S}}}(L\circ s'))\circ t')\\
&=&t\bullet_{\mu_{\mathsf{T}}} (f_L(s)\circ f_L(L\circ s')\circ t')\\
&=&t\bullet_{\mu_{\mathsf{T}}} D(J,f_L(s))(f_L(L\circ s')\circ t')\\
&=&D(J,f_R(s))(t)\bullet_{\mu_{\mathsf{T}}}(f_L(L\circ s')\circ t')\\
&=&(f_R(s)\circ t)\bullet_{\mu_{\mathsf{T}}}G(L,R)(u')\\
&=&u\bullet_{\mu_{\mathsf{T}}}G(L,R)(u').
\end{array}
\end{equation}
Consequently, $(G(L,R),D(L,R))\in \mathit{TrHull}(\mathsf{Conv}_J(\mathsf{D}))$. By the concreteness assumption then so exists $f^M(L,R)\in \mathit{Mult}_{\mathbb{D}}(\mathsf{T})$ such that $D(J,f^M(L,R))=(G_f(L,R),D_f(L,R))$, that is, for each $u\in D(J,T)$, $f^M_L(L,R)\circ u=G_f(L,R)(u)$ and $f^M_R(L,R)\circ u=D_f(L,R)(u)$.

Let us assume that for each $(L,R)\in \mathit{Mult}_{\mathbb{C}}(\mathsf{S})$ is chosen $f^M(L,R)\in \mathit{Mult}_{\mathbb{D}}(\mathsf{T})$ as above. One then defines a map $f^M\colon (L,R)\mapsto f^M(L,R)$. Let $a\in C(I,S)$ and let $u=f_L(s)\circ t\in D(J,T)$ for $s\in C(I,S)$ and $t\in D(J,T)$. Then,
\begin{equation}
\begin{array}{lll}
f_L^M(\mathcal{M}_{\mathbb{C},\mathsf{S}}(a))\circ u&=&G(L_a,R_a)(f_L(s)\circ t)\\
&=&f_L(L_a\circ s)\circ t\\
&=&f_L(a\bullet_{\mu_S} s)\circ t\\
&=&f_L(a)\circ f_L(s)\circ t\\
&=&f_L(a)\circ u.
\end{array}
 \end{equation}
 By non-degeneracy of $\mathsf{Conv}_J(\mathsf{T})$, $f_L^M(\mathcal{M}_{\mathbb{C},\mathsf{S}}(a))=f_L(a)$. Similarly, $f_R^M(\mathcal{M}_{\mathbb{C},\mathsf{S}}(a))\circ u=f_R(a)\circ u$. One so obtains that 
 $f^M\circ \mathcal{M}_{\mathbb{C},\mathsf{S}}=f$. 
 
 Let $(L,R),(L',R')\in \mathit{Mult}_{\mathbb{C}}(\mathsf{S})$ and let $u,t\in D(J,T)$, $s\in C(I,S)$ such that $f_L(s)\circ t=u$. Then, 
 \begin{equation}
 \begin{array}{lll}
 f_L^M((L',R')\star (L,R))\circ u\\
 &=&f_L^M(L'\circ L,R\circ R')\circ f_L(s)\circ t\\
 &=&f_L(L'\circ L\circ s)\circ t\\
 &=&f_L^M(L',R')\circ f_L(L\circ s)\circ t\\
 &=&f_L^M(L',R')\circ f_L^M(L,R)\circ f_L(s)\circ t\\
 &=&f_L^M(L',R')\circ f_L^M(L,R)\circ u.
 \end{array}
 \end{equation}
 By non-degeneracy of $\mathsf{Conv}_J(\mathsf{T})$, it then follows that $ f_L^M((L',R')\star (L,R))=f_L^M(L',R')\circ f_L^M(L,R)$. Similarly, $f_R^M((L',R')\star (L,R))=f_R^M(L,R)\circ f_R^M(L',R')$. Consequently, $f^M((L',R')\star (L,R))=f^M(L',R')\star f^M(L,R)$. Finally, since it quite clear that $f^M(id_S,id_S)=(id_T,id_T)$, $f^M \in \mathbf{Mon}(\mathsf{Mult}_{\mathbb{C}}(\mathsf{S}),\mathsf{Mult}_{\mathbb{D}}(\mathsf{T}))$. 
   
The  remaining assertion is checked by direct inspection. 
\end{proof}

\begin{definition}
Let $\mathbf{IntSem}_{nd}$ be the subcategory of $\mathbf{IntSem}$ consisting of those  objects $(\mathbb{C},\mathsf{S})$ with $\mathsf{Conv}_{\mathbb{C}}(\mathsf{S})$ globally idempotent and non-degenerate, and  morphisms $(\mathbb{F},f)\colon (\mathbb{C},\mathsf{S})\to (\mathbb{D},\mathsf{T})$ such that $\mathsf{Conv}(\mathbb{F},f)\colon \mathsf{Conv}_{\mathbb{C}}(\mathsf{S})\to \mathsf{Conv}_{\mathbb{D}}(\mathsf{T})$ is non-degenerate.

Let ${}_{conc}\mathbf{IntSem}_{nd}$ be the full subcategory of $\mathbf{IntSem}_{nd}$ spanned by the concrete semigroup objects. 
\end{definition}

\begin{remark}
By definition, $\mathsf{Conv}\colon \mathbf{IntSem}\to \mathbf{Sem}$ restricts and co-restricts as a functor $\mathsf{Conv}_{nd}\colon \mathbf{IntSem}_{nd}\to \mathbf{Sem}_{nd}$. 
\end{remark}

\begin{proposition}
There is a functor $\mathit{Mult}\colon {}_{conc}\mathbf{IntSem}_{nd}\to \mathbf{Mon}$ such that $\mathit{Mult}(\mathbb{C},\mathsf{S})=\mathsf{Mult}_{\mathbb{C}}(\mathsf{S})$.
\end{proposition}
\begin{proof}
Let $(\mathbb{F},f)\in {}_{conc}\mathbf{IntSem}_{nd}((\mathbb{C},\mathsf{S}),(\mathbb{D},\mathsf{S}))$.  It suffices to define $\mathit{Mult}(\mathbb{F},f):=(\mathcal{M}_{\mathbb{D},\mathsf{T}}\circ \mathsf{Conv}(\mathbb{F},f))^M\colon \mathsf{Mult}_{\mathbb{C}}(\mathsf{S})\to \mathsf{Mult}_{\mathbb{D}}(\mathsf{T})$ using Theorem~\ref{thm:extension_for_internal_multipliers}. \end{proof} 

Let $\mathit{Conc}:=(\mathit{Conc}_{\mathbb{C}}(\mathsf{S}))_{(\mathbb{C},\mathsf{S})\in Ob {}_{conc}\mathbf{IntSem}_{nd}}$. Using Theorem~\ref{thm:extension_for_internal_multipliers}, one easily shows  that  $\mathit{Conc}$ is a natural transformation $\mathit{Mult}\Rightarrow \mathit{TrHull}\circ \mathsf{Conv}_{nd}$ from ${}_{conc}\mathbf{IntSem}_{nd}$ to $\mathbf{Mon}$.  As a consequence of the above lemma, the following diagram commutes up to the natural epimorphism $\mathit{Conc}$.
\begin{equation}
\xymatrix@R=1em{
\mathbf{IntSem}_{nd} \ar[r]^{\mathit{TrHull}\circ \mathsf{Conv}_{nd}}& \mathbf{Mon}\\
\ar@{^{(}->}[u]{}_{conc}\mathbf{IntSem}_{nd} \ar[ru]_{\mathit{Mult}}& 
}
\end{equation}


\begin{thebibliography}{99}

\bibitem{Borceux2} 
Borceux, F. (1994). {\em Handbook of Categorical Algebra: Volume 2, Categories and Structures}. Cambridge: Cambridge University Press.

\bibitem{Bohm}
B\"ohm, G., and Lack, S. "A category of multiplier bimonoids." Applied Categorical Structures, vol.~25 (2017), pp.~279--301.
\bibitem{Busby}
Busby, R.~C. "Double centralizers and extensions of $C*$-algebras." Transactions of the American Mathematical Society, vol.~132 (1968), pp.~79-99.

\bibitem{Clifford}
Clifford, A. H., and Preston, G. B. (1961). {\em The algebraic theory of semigroups}. AMS surveys.
\bibitem{MacLane}
Mac Lane, S. (2013). {\em Categories for the working mathematician (Vol. 5)}. Springer Science \& Business Media.

\bibitem{Palmer}
Palmer, T. W. (1994). {\em Banach Algebras and the General Theory of $*$-Algebras: Volume 1, Algebras and Banach Algebras}. Cambridge: Cambridge university press.

\bibitem{Petrich}
Petrich, M. "The translational hull in semigroups and rings." Semigroup Forum, vol.~1 (1970), pp.~283-360. 

\bibitem{Poinsot}
Poinsot, L. "Hilbertian (function) algebras." Communications in Algebra, vol.~48 (2020), pp.~961-991.

\bibitem{Ryan}
Ryan, R. A. (2013). {\em Introduction to Tensor Products of Banach Spaces}. New York, USA: Springer Science \&
Business Media.

\bibitem{Street}
Street, R. (2007). {\em Quantum groups}. Cambridge: Cambridge University Press.

\bibitem{Daele}
Van Daele, A. "Multiplier Hopf algebras." Transactions of the American Mathematical Society, vol.~342 (1994), pp.~917-932.

\bibitem{Wang}
Wang, J.-K. "Multipliers of commutative Banach algebras." Pacific J. Math., vol.~11 (1961) no.~4, pp.~1131--1149. 


%


\end{thebibliography}
\end{document}